\documentclass[11pt]{amsart}
\usepackage{graphicx}
\usepackage{latexsym,amsmath,amsfonts,amscd, amsthm}
\usepackage{changebar}
\usepackage{color}
\usepackage{bm}
\usepackage{tikz}
\usepackage{multirow}
\usetikzlibrary{arrows,backgrounds,snakes}
\usepackage{epsfig,subfigure}

\newtheoremstyle{plainNoItalics}{}{}{\normalfont}{}{\bfseries}{.}{ }{}
\theoremstyle{plain}
\newtheorem{thm}{Theorem}[section]
\theoremstyle{plainNoItalics}

\newtheorem{defn}[thm]{Definition}

\newtheorem{prop}[thm]{Proposition}
\newtheorem{exa}[thm]{Example}

\newcommand{\f}{\frac}

\newcommand{\beq}{\begin{equation}}
\newcommand{\eeq}{\end{equation}}
\newcommand{\beqa}{\begin{eqnarray}}
\newcommand{\eeqa}{\end{eqnarray}}
\newcommand{\bit}{\begin{itemize}}
\newcommand{\eit}{\end{itemize}}
\newcommand{\bedef}{\begin{defn}}
\newcommand{\edefn}{\end{defn}}
\newcommand{\bpro}{\begin{prop}}
\newcommand{\epro}{\end{prop}}

\newcommand{\iL}{{i-\frac{1}{2}}}
\newcommand{\iR}{{i+\frac{1}{2}}}
\newcommand{\jL}{{j-\frac{1}{2}}}
\newcommand{\jR}{{j+\frac{1}{2}}}
\newcommand{\pt}{\partial_t}

\newcommand{\bfx}{{\bf x}}
\newcommand{\bfv}{{\bf v}}

\newcommand{\bfE}{{\bf E}}

\newcommand{\bfU}{{\bf U}}

\newcommand\ds{ \displaystyle }

\def\signTX{\bigskip\bigskip\hspace{80mm}
\vbox{{\sc Tao Xiong\par\vspace{2mm}
School of Mathematical Sciences, \par 
Fujian Provincial Key Laboratory of
Math. Mod. \& HPSC,
\par
 Xiamen University, Xiamen, Fujian, P.R. China, 361005.\par
\vspace{1.5mm}
Universit\'e de Toulouse III  \par
UMR5219, Institut de Math\'ematiques de Toulouse,\par
118, route de Narbonne\par
F-31062 Toulouse cedex, FRANCE
\par\vspace{1.5mm}e-mail:} txiong@xmu.edu.cn}}
\begin{document}

\title[High order bound preserving finite difference linear scheme]{A high order bound preserving finite difference linear scheme for incompressible flows} 

\author{Tao Xiong}

\maketitle

\begin{abstract}
We propose a high order finite difference linear scheme combined with
a high order bound preserving maximum-principle-preserving (MPP) flux
limiter to solve the incompressible flow system. For such problem with highly oscillatory structure but not strong shocks, our approach seems
to be  less dissipative and much less costly than a WENO type scheme,
and has high resolution due to a Hermite reconstruction. Spurious numerical oscillations can be controlled by the MPP flux limiter. Numerical tests are performed for the Vlasov-Poisson system, the 2D guiding-center model and the incompressible Euler system. The comparison between the linear and WENO type schemes will demonstrate the good performance of our proposed approach. 

 \end{abstract}

\vspace{0.1cm}


\tableofcontents

\section{Introduction}
\label{sec1}
\setcounter{equation}{0}
\setcounter{figure}{0}
\setcounter{table}{0}

%%%%%%%%%%%%%%%%%%%%
%%
%%%%%%%%%%%%%%%%%%%%

In this paper, we are interested in the numerical approximation of
incompressible transport equations as  
\beq
\left\{
\begin{array}{l}
\displaystyle\frac{\partial \rho}{\partial t} \,+\,  \bfU \cdot \nabla \rho \,=\, 0, 
\\
\,
\\
{\rm div} \, \bfU \, = \, 0,
\end{array}
\label{eq:2d}
\right.
\eeq
where $\bfU$ represents the advection field and $\rho$ is a nonnegative density.

The typical example of application is the well known Vlasov-Poisson
(VP) system arising in collisionless plasma physics. It describes the time evolution of particles under the effects of self-consistent electrostatic field and reads
\beq
\frac{\partial f}{\partial t} + \bfv \cdot \nabla_\bfx f + \bfE \cdot \nabla_\bfv f = 0,
\label{eq:vlasov}
\eeq
$f\, :=\, f(t,\bfx,\bfv)$ is the distribution function in the phase space $(\bfx,\bfv)\in \mathbb{R}^d\times\mathbb{R}^d, \,\, d=1,2,3$. $\bfE\, :=\, \bfE(t,\bfx)$ is the electric field,
which can be determined by the Poisson's equation from an electric potential function $\Phi(t,\bfx)$
\beq
\bfE(t,\bfx) \,=\, - \nabla_\bfx\, \Phi(t,\bfx), \qquad -\Delta_\bfx \,\Phi(t,\bfx) \,=\, \rho(t,\bfx).
\label{eq:poisson}
\eeq
The charge density $\rho(t,\bfx)$ is defined as
\[
\rho(t,\bfx) \, = \, \int_{\mathbb{R}^d}f(t,\bfx,\bfv)\,d\bfv.
\]

Another example is the two dimensional guiding-center model, which describes the evolution of the charge density $\rho$ in a highly magnetized plasma
in the transverse plane of a tokomak, is given by
\beq
\left\{
\begin{array}{l}
\displaystyle\frac{\partial \rho}{\partial t} \,+\,  \bfU \cdot \nabla \rho \,=\, 0, 
\\
\,
\\
-\Delta \, \Phi \, = \, \rho,
\end{array}
\label{eq:2dgc}
\right.
\eeq
where $\bfU \,=\, (-\Phi_y, \Phi_x)$ is a divergence free
velocity. The two dimensional guiding center model
can also be referred as an asymptotic model of the VP system by averaging in the velocity phase
space, for details, see \cite{madaule2013semi}. We notice that the guiding-center model \eqref{eq:2dgc} is in the same form
as the two dimensional incompressible Euler equations in the vorticity stream function formulation, which
describes the evolution of vortices in fluid hydrodynamics.

For all the models we mention above, they all have an transport equation coupled with a Poisson's equation for the advection velocity,
and moreover, the advection velocity is divergence free. In the following, we refer them
as incompressible flow models.

Many numerical schemes have been proposed for solving these models, especially recently,
high order schemes are very attractive due to their high resolutions for such problems with
rich solution structures. For example, deterministic methods, there are finite difference, finite volume and finite element Eulerian methods \cite{christlieb2016weno,yang2014conservative,xiong2013parametrized,zhang2012maximum,carrillo2007nonoscillatory,filbet2003comparison,filbet2001conservative,zhou2001numerical,fijalkow1999numerical,zaki1988finite1,zaki1988finite}, semi-Lagrangian methods \cite{besse2017adaptive, xiong2016conservative,cai2016conservative,qiu2016high,xiong2014high,guo2013hybrid,crouseilles2011discontinuous,qiu2011conservative,qiu2011positivity,rossmanith2011positivity,crouseilles2010conservative,qiu2010conservative,besse2003semi,sonnendrucker1999semi}, and discontinuous Galerkin finite element methods \cite{guo2016adaptive,zhu2016h,madaule2014energy,cheng2013study,de2012discontinuous,heath2012discontinuous,ayuso2009discontinuous}, also see many other references therein. However, due
to the highly oscillatory structure of such problems, linear type schemes
for these problems would show significant spurious numerical oscillations, which might get worse
with increased orders. Weighted essentially nonoscillotry (WENO) reconstruction, which was originally developed in the presence of both shocks and small fine structures for fluid hydrodynamics of hyperbolic conservation laws, see for example \cite{shu1998essentially}, is frequently adopted in most of the finite difference and finite volume Eulerian or semi-Lagrangian methods mentioned above. For WENO type schemes, we can often observe the good performance of the WENO reconstruction on suppressing spurious oscillations \cite{qiu2016high,xiong2014high,yang2014conservative,xiong2013parametrized,zhang2012maximum}, which can also be seen from our examples in the numerical 
section. However, for incompressible flow problems, their solutions are highly oscillatory but without strong discontinuities. We might expect excessive usage of the WENO reconstruction, which is computationally more expensive than a pure linear type scheme and too dissipative for certain classes of problems \cite{jia2015spectral}, especially for high dimensional problems with long time simulations. Although hybrid approaches of coupling a linear scheme and a WENO scheme can save some cost from the WENO reconstruction, for example the very recent work \cite{don2016hybrid} and references therein, they still focus on problems with strong discontinuities and WENO may not be avoided. 

In this paper, we propose to solve the incompressible flow problems with a high order linear scheme without WENO reconstruction. Linear schemes have the following several good properties: (1) less dissipative: for example, it preserves the $L^2$ norm (also energy and entropy for the VP system) better than the WENO type scheme; (2) less costly and easier implementation: without WENO reconstruction, it saves a lot computational cost and can be easily implemented, especially when extended to high dimensional problems; (3) higher resolution if with a Hermite reconstruction. We adopt the scheme in \cite{yang2014conservative} but with a Hermite linear reconstruction.
In order to control the spurious numerical oscillations due to a linear type scheme, we seek to combine it with a newly developed high order bound
preserving maximum-principle-preserving (MPP) flux limiter. The MPP flux limiter was first proposed by Xu et. al. \cite{xu2014parametrized,liang2014parametrized}, and then improved by Xiong et. al. \cite{xiong2013parametrized,xiong2014high}. It can be seen as a very weak limiter, which just pulls the numerical overshootings and undershootings back to its physical range, without excessive dissipating the solution within the range. Moreover, there is no further time step restriction on this MPP flux limiter from the original scheme for linear stabilities. 
We refer to \cite{xu2016bound} for a review of recent works on this bound preserving high order flux limiter. The coupling of the linear scheme and the MPP flux limiter keeps the original high order accuracy while maintaining a large CFL number. Due to the highly oscillatory but non discontinuous solutions of incompressible flow problems, the MPP flux limiter can serve as a necessary auxiliary tool for the linear scheme. The extra work from applying the MPP flux limiter at each final stage of a multi-stage Runge-Kutta time discretization, is much less than the WENO reconstruction. Numerical experiments, especially the bump-on-tail instability of long time simulation and the Kelvin-Helmoholtz instability of the 2D guiding-center model are as benchmark tests, will be performed to demonstrate the good performance of our proposed approach. 

The rest of the paper is organized as follows. In Section 2, we will describe the conservative finite difference
scheme with both linear and WENO reconstructions, for the completeness of comparison in the numerical section. The bound preserving MPP flux limiter will also be briefly reviewed. In Section 3, numerical experiments including the VP system, the Kelvin-Helmholtz instability of the 2D guiding center model and the incompressible Euler system will be studied. Conclusions are made in Section 4.

%%%%%%%%%%%%%%%%%%%
%
%%%%%%%%%%%%%%%%%%%
\section{Finite difference scheme with Hermite reconstruction}
\label{sec4}
\setcounter{equation}{0}
\setcounter{figure}{0}
\setcounter{table}{0}

In this section, we will describe our scheme to solve the VP system \eqref{eq:vlasov}-\eqref{eq:poisson}, the 2D guiding-center model \eqref{eq:2dgc} as well as the incompressible Euler system.
Here we just consider $d=1$ for the VP system, in common we will have a 2D transport
equation with a divergence free velocity, coupled with a Poisson's type equation, which is 1D for
the VP system, and 2D for the other two. In the following, we will take the 2D
guiding-center model \eqref{eq:2dgc} as an example to describe our schemes. The other two models can be applied similarly. 

We propose a finite difference scheme with a Hermite linear reconstruction for solving the 2D conservative transport equation. We adopt the scheme developed by Filbet and Yang \cite{yang2014conservative},  which has a Hermite WENO reconstruction. In the following, we will briefly recall the 2D conservative finite difference scheme and describe both the Hermite linear and WENO reconstructions. We note that the smooth indicators for the Hermite WENO reconstruction are modified as compared to \cite{yang2014conservative}. The Poisson's equation for the electric potential function $\Phi$ will be solved by fast Fourier transform (FFT) for periodic boundary condition on an interval in one-dimension (1D) or periodic boundary conditions on a rectangular domain in two-dimension (2D), which will be omitted here. We refer to \cite{yang2014conservative} for more details.

\subsection{Conservative finite difference scheme}
We consider the 2D  transport equation in a conservative form 
$$
\pt \rho \,+\, {\rm div}_\bfx \left( {\bf U} \rho \right) \, = \, 0,
$$
with $\bfU = \bfU(t,\bfx)$ such that ${\rm div}_\bfx{\bf U} =0$ and $\bfx=(x,y)$. For simplicity, we assume a uniform discretization of the computational domain $[x_{min}, x_{max}] \times [y_{min}, y_{max}]$ with $N_x\times N_y$ grid points
\[
\begin{array}{l}
x_{min} \,=\, x_0 < x_1 < \cdots < x_{N_x-1} < x_{N_x} \,=\, x_{max}, 
\\
\,
\\
y_{min} \,=\, y_0 < y_1 < \cdots < y_{N_y-1} < y_{N_y} \,=\,y_{max},
\end{array}
\]
where the mesh sizes are $\Delta x = x_{i+1}-x_i$ and $\Delta y = y_{j+1} - y_j$ for $0\le i \le N_x, \,\, 0\le j \le N_y$. A conservative finite difference scheme with Euler forward time discretization is defined as follows:
\beq
\rho^{n+1}_{i,j} \,=\, \rho^n_{i,j} - \Delta t \left( \frac{\hat H_{\iR,j}-\hat H_{\iL,j}}{\Delta x} + \frac{\hat G_{i,\jR}-\hat G_{i,\jL}}{\Delta y}\right),
\label{eq:1st}
\eeq
where the time step is $\Delta t = t^{n+1}-t^n$. $\rho^n_{i,j}$ is the numerical value at time level $t^n$ on the grid point $(x_i,y_j)$. $\hat H_{\iR,j}$, $\hat G_{i,\jR}$ are the numerical fluxes in the $x$ and $y$ directions respectively. 

\subsection{Hermite linear reconstruction}
For a finite difference scheme \eqref{eq:1st}, the numerical fluxes $\hat H_{\iR,j}$ and $\hat G_{i,\jR}$ are reconstructed dimension by dimension. Here we will take a 1D transport equation to illustrate on how to obtain a flux by a Hermite linear reconstruction. $\hat H_{\iR,j}$ and $\hat G_{i,\jR}$ are obtained in this way along each of its own direction. The Hermite linear reconstruction is what we propose in this paper. A corresponding Hermite WENO reconstruction will be described in the next subsection, which is used for the comparison in the numerical section.

Let us consider a prototype 1D conservative transport equation
\begin{equation}
\partial_t\rho \,+\,\partial_x \,(U \rho) \,=\, 0,
\label{eq:1Dtrans}
\end{equation}
with velocity $U=U(t,x)$ and a uniform discretization with mesh size $\Delta x=x_{k+1}-x_k$ for $0\le k < N_x$. A conservative finite difference scheme for \eqref{eq:1Dtrans} can be written as
\[
\rho^{n+1}_{i} \,=\, \rho^n_{i} - \frac{\Delta t}{\Delta x} \left( \hat h_{i+\frac12}- \hat h_{i-\frac12} \right),
\]
where $\rho^n_i$ is the numerical point value at time level $t^n$ on the grid point $x_i$. 
$\hat h_{i+\frac12}$ can be chosen as an upwind numerical flux, which is 
\[
\hat h_{i+\frac12} \,=\,
\left\{
\begin{array}{l}
\displaystyle h^-_{i+\frac12}, \qquad \text{ if }\frac{U^n_i + U^n_{i+1}}{2} > 0, 
\\
\,
\\
h^+_{i+\frac12}, \qquad \text{ otherwise }.
\end{array}
\right.
\]
$h^\pm_{i+\frac12}$ are fluxes reconstructed from $\{h^n_i=U^n_i \rho^n_i\}_i$ by a Hermite linear reconstruction, from the left and right sides of $x_{i+\frac12}$ respectively, where $U^n_i$ is the numerical velocity approximating $U(t^n, x_i)$. For simplicity, we drop the superscript $n$ for $h^n_i$ and $h^-_{\iR}$ is simply reconstructed by
\beq
h^-_{\iR} \, = \, \frac{1}{27}\left(-8h_{i-1}+19h_i+19h_{i+1}+3G'_{i-3/2}-6G'_{i+3/2}\right),
\label{eq:linear5}
\eeq
where the derivative of the primitive function $G'_{i+\frac12}$ is given by a 6th order central difference approximation
\beq
G'_{i+\frac12} \,=\, \frac{1}{60}\Big[(h_{i+3}-h_{i-2})-8(h_{i+2}-h_{i-1})+37(h_{i+1}-h_{i})\Big],
\label{eq:linear5G}
\eeq
$h^+_{\iR}$ can be obtained in mirror symmetric with respect to $x_\iR$.

\subsection{Hermite WENO reconstruction}

Now we will describe a fifth order Hermite WENO reconstruction to compute $h^-_{i+\frac12}$ corresponding to \eqref{eq:linear5}. The procedure is outlined as follows. $h^+_{i+\frac12}$ can also be obtained in mirror symmetric with respect to $x_{i+\frac12}$. We drop the superscript $n$ for $h^n_i$ and we have
\beq
h^-_{i+\frac12} \,=\, \omega_l h_l(x_{i+\frac12}) + \omega_c h_c(x_{i+\frac12}) + \omega_r h_r(x_{i+\frac12}).
\label{eq:weno5}
\eeq
The three polynomials $h_l(x)$, $h_c(x)$ and $h_r(x)$ evaluating at $x_{i+\frac12}$ are
$$
\left\{
\begin{array}{l}
\ds h_l(x_{i+\frac12}) = -2 h_{i-1} + 2 h_{i} + G'_{i-\frac32}, \\ \,\\
\ds h_c(x_{i+\frac12}) = \frac{-h_{i-1} + 5 h_{i} + 2 h_{i+1}}{6}, \\ \,\\
\ds h_r(x_{i+\frac12}) = \frac{h_i + 5 h_{i+1} -2 G'_{i+\frac32}}{4},
\end{array}
\right.
$$
whereas the weights  $\omega_l$, $\omega_c$ and $\omega_r$ are the nonlinear WENO weights and determined according to the smoothness indicators
\[
\omega_k \,=\, \frac{\alpha_k}{\alpha_l + \alpha_c + \alpha_r}, \quad \alpha_k \,=\, \frac{c_k}{(\epsilon +\beta_k)^2}, \quad k = l, \, c, \, r.
\]
In order to match \eqref{eq:linear5}, the linear coefficients are $c_l = 1/9$ and $c_c = c_r = 4/9$. The small parameter $\epsilon = 10^{-6}$ is to avoid the denominator to be $0$.

To evaluate the smooth indicators $\beta_l$, $\beta_c$ and $\beta_r$, we measure them on the cell $[x_{i-\frac12}, x_{i+\frac12}]$ instead of $[x_i, x_{i+1}]$ as in \cite{yang2014conservative}. In this way, the smooth indicators are symmetric with respect to $x_i$, as we can see below:
\begin{align*}
\beta_l &\,=\, \int_{x_{i-\frac12}}^{x_{i+\frac12}} \Delta x(h'_l(x))^2 + \Delta x^3 (h''_l(x))^2 dx \\
&\,=\, \frac{13}{16}s_1^2+\frac{3}{16}(s_1-4 s_2)^2, \,\, \text{ with } s_1=h_{i-1}-h_i, \quad s_2=-3h_{i-1}+h_0+G'_{i-\frac32} , \\
\beta_c &\,=\, \int_{x_{i-\frac12}}^{x_{i+\frac12}} \Delta x(h'_c(x))^2 + \Delta x^3 (h''_c(x))^2 dx \\
&= \frac{1}{4}s_1^2+\frac{13}{12}s_2^2, \,\, \text{ with } s_1=h_{i+1}-h_{i-1}, \quad s_2=h_{i+1}-2h_0+h_{i-1}, \\
\beta_r &\,=\, \int_{x_{i-\frac12}}^{x_{i+\frac12}} \Delta x(h'_r(x))^2 + \Delta x^3 (h''_r(x))^2 dx \\
&= \frac{13}{16}s_1^2+\frac{3}{16}(s_1-4 s_2)^2, \,\, \text{ with } s_1=h_{i+1}-h_i, \quad s_2=-3h_{i+1}+h_0+G'_{i+\frac32} .
\end{align*}

\subsection{High order Runge-Kutta time discretization}
The first order Euler forward time discretization \eqref{eq:1st} can be generalized to high order Runge-Kutta (RK) time discretization \cite{shu1988efficient}. For example, if we write \eqref{eq:1st} in the following form
\[
\rho^{n+1} \,=\, \rho^n +  \Delta t \, L(\rho^n),
\]
here the subscripts for $\rho$ are dropped and the operator $L$ denotes the spatial discretization, then a 4th order RK time discretization can be written as
\beq
\left\{
\begin{array}{l}
\rho^{(1)} \,=\, \rho^n + \frac12 \, \Delta t \, L(\rho^n), 
\\
\,
\\
\rho^{(2)} \,=\, \rho^n + \frac12 \, \Delta t \, L(\rho^{(1)}), 
\\
\,
\\
\rho^{(3)} \,=\, \rho^n + \Delta t\, L(\rho^{(2)}),
\\
\,
\\
\rho^{n+1} \,=\,  \rho^n + \frac{1}{6}\,\Delta t\, \left(L(\rho^n) + 2 L(\rho^{(1)}) +2 L(\rho^{(2)}) + L(\rho^{(3)})\right),
\end{array}
\right.
\label{eq:rk4}
\eeq
with a CFL number $\frac23$ for linear stability \cite{shu1988efficient}. The last stage of \eqref{eq:rk4} can be written in the same form as \eqref{eq:1st}, which is 
\beq
\rho^{n+1}_{i,j} \,=\, \rho^n_{i,j} - \Delta t \left( \frac{\hat H^{rk}_{\iR,j}-\hat H^{rk}_{\iL,j}}{\Delta x} + \frac{\hat G^{rk}_{i,\jR}-\hat G^{rk}_{i,\jL}}{\Delta y}\right),
\label{eq:high}
\eeq
where the numerical flux $\hat H^{rk}_{\iR,j}$ is accumulated from all formal stages, that is
\[
\hat H^{rk} \,= \, \frac16\left(\hat H^{(0)} + 2 \hat H^{(1)} + 2\hat H^{(2)} + \hat H^{(3)}\right),
\]
where $\hat H^{(l)}$ is the numerical flux at the corresponding $l$-th stage, and the subscript $(\iR,j)$ is dropped for clarity. $\hat G^{rk}_{i,\jR}$ is defined similarly.

%\begin{rem}
%	For the  two dimensional transport equation with Dirichlet boundary conditions, and/or if the physical domain is not rectangular, e.g. a disk, an inverse Lax -Wendroff method \cite{filbet2013inverse} can be used to deal with the curved boundary conditions inside a larger rectangular computational domain.
%\end{rem}

\subsection{Bound preserving MPP flux limiter}

For a high order bound preserving limiter, we adopt the parametrized maximum-principle-preserving (MPP) flux limiter developed in \cite{xiong2013parametrized,xiong2014high}. The MPP flux limiter is only applied at the final stage of \eqref{eq:rk4}, that is \eqref{eq:high}. Due to the same form of \eqref{eq:1st} and \eqref{eq:high}, the MPP flux limiter is defined as a convex combination of a first order monotone flux and the high order flux in \eqref{eq:high}. In the following, we first 
describe a first order monotone MPP scheme for a general incompressible flow system, then we recall
a specific first order monotone MPP scheme developed in \cite{xiong2013parametrized} for the incompressible flow system with the Poisson's equation.

For a general incompressible flow system with divergence free condition
\beq
\label{eq:ins}
\begin{cases}
	\pt \rho + {\rm div}_{\bfx} \,(\bfU\rho) \, = \, 0, \\
	{\rm div}_{\bfx} \,\bfU \, = \, 0,
\end{cases}
\eeq
a first order scheme can be defined as
\beq
\label{eq:ins:ord1}
\begin{cases}
\rho^{n+1}_{i,j} \, = \, \rho^n_{i,j} - \lambda_x  \left(U^-_{\iR,j}\rho^n_{i,j} + U^+_{\iR,j}\rho^n_{i+1,j}-U^-_{\iL,j}\rho^n_{i-1,j}-U^+_{\iL,j}\rho^n_{i,j}\right) \\\hspace{2.1cm}-\lambda_y \left(U^-_{i,\jR}\rho^n_{i,j}+U^+_{i,\jR}\rho^n_{i,j+1}-U^-_{i,\jL}\rho^n_{i,j-1}-U^+_{i,\jL}\rho^n_{i,j}\right), \\
\frac{1}{\Delta x}\left(U^-_{\iR,j} + U^+_{\iR,j}-U^-_{\iL,j}+U^+_{\iL,j}\right) + \frac{1}{\Delta y}\left(U^-_{i,\jR}+U^+_{i,\jR}-U^-_{i,\jL}-U^+_{i,\jL}\right) \, = \, 0,
\end{cases}
\eeq
where $\lambda_x = \Delta t/\Delta x$ and $\lambda_y = \Delta t /\Delta y$.
If the scheme \eqref{eq:ins:ord1} is a consistent discretization to \eqref{eq:ins}, we have the following statement:
\begin{prop}
For a first order consistent scheme \eqref{eq:ins:ord1} solving the incompressible system \eqref{eq:ins},
it is monotone and MPP, if the time step is small enough, e.g., 
\beq
\label{ord1:dt}
\Delta t \le \frac{1}{2\max|\bfU|}\frac{\Delta x\Delta y}{\Delta x + \Delta y}
\eeq
and the coefficients have signs that
\beq
\label{ord1:sign}
U^-_{\iR,j} \ge 0, \quad U^+_{\iR,j} < 0, \quad U^-_{i,\jR} \ge 0, \quad U^+_{i,\jR} < 0,
\eeq
for all $i,j$.
\end{prop}
\begin{proof}
We can rewrite the first equation in the scheme \eqref{eq:ins:ord1} to be
\beq
\label{eq:ins:ord1-1}
\rho^{n+1}_{i,j} = \alpha_{i,j}\rho^n_{i,j} + \alpha_{i+1,j}\rho^n_{i+1,j}+\alpha_{i-1,j}\rho^n_{i-1,j}+\alpha_{i,j+1}\rho^n_{i,j+1}+\alpha_{i,j-1}\rho^n_{i,j-1},
\eeq
with
\[
\begin{cases}
\alpha_{i,j} = 1- \lambda_x \left(U^-_{\iR,j}-U^+_{\iL,j}\right) - \lambda_y \left(U^-_{i,\jR}-U^+_{i,\jL}\right), \\
\alpha_{i+1,j} = -\lambda_x U^+_{\iR,j}, \qquad \alpha_{i-1,j} = \lambda_x U^-_{\iL,j}, \\
\alpha_{i,j+1} = -\lambda_y U^+_{i,\jR}, \qquad \alpha_{i,j-1} = \lambda_y U^-_{i,\jL}.
\end{cases}
\]
It is easy to check that $\alpha_{i,j},\alpha_{i\pm 1,j }, \alpha_{i,j\pm1}$ are all positive if the two conditions
\eqref{ord1:dt} and \eqref{ord1:sign} are satisfied. Moreover,
\[
\alpha_{i,j} + \alpha_{i+1,j}+\alpha_{i-1,j}+\alpha_{i,j+1}+\alpha_{i,j-1}\, = \, 1,
\]
due to the enforced discrete divergence free condition in \eqref{eq:ins:ord1}. Since \eqref{eq:ins:ord1-1} is a convex combination, so it is monotone 
and MPP.
\end{proof}
A specific first order monotone MPP scheme depends on how to choose $U^\pm_{\iR,j}$ and $U^\pm_{i,\jR}$. By using a potential function $\Phi$, where $\bfU = (-\Phi_y,\Phi_x)$, a first order monontone MPP scheme with Lax-Friedrichs flux defined in \cite{xiong2013parametrized} is to take:
\beq
\begin{cases}
& U^-_{\iR,j}  = \frac12\left(\alpha_x-\frac{\Phi_{i,j+1}-\Phi_{i,j}}{\Delta y} \right),\qquad  U^+_{\iR,j} =\frac12\left(-\alpha_x-\frac{\Phi_{i+1,j}-\Phi_{i+1,j-1}}{\Delta y} \right), \\
&U^-_{i,\jR}  = \frac12\left(\alpha_y+\frac{\Phi_{i+1,j}-\Phi_{i,j}}{\Delta x} \right), \qquad  U^+_{i,\jR} =\frac12\left(-\alpha_y+\frac{\Phi_{i,j+1}-\Phi_{i-1,j+1}}{\Delta x} \right),
\end{cases}
\eeq
where 
\[
\alpha_x = \max_{i,j}\left| \frac{\Phi_{i,j+1}-\Phi_{i,j}}{\Delta y}\right|, \qquad \alpha_y = \max_{i,j} \left|\frac{\Phi_{i+1,j}-\Phi_{i,j}}{\Delta x} \right|.
\]
In this paper, the potential function $\Phi$ is obtained by solving the Poisson's equation $\Delta \Phi = \rho$ with FFT.

Now to describe the MPP flux limiter, we write the first order monotone MPP scheme in a conservative flux difference form
\beq
\rho^{n+1}_{i,j} \,=\, \rho^n_{i,j} - \lambda_x \left( \hat h_{\iR,j}- \hat h_{\iL,j} \right) - \lambda_y\left(\hat g_{i,\jR} - \hat g_{i,\iL}\right),
\eeq
with
\beq
\left\{
\begin{array}{l}
\hat{h}_{i+\frac{1}{2},j}= \frac12\left(\alpha_x-\frac{\Phi_{i,j+1}-\Phi_{i,j}}{\Delta y}\right)\rho^n_{i,j}+\frac12\left(-\alpha_x-\frac{\Phi_{i+1,j}-\Phi_{i+1,j-1}}{2\Delta y}\right)\rho^n_{i+1,j},
\\
\,
\\
\hat{g}_{i,j+\frac{1}{2}}=\frac12\left(\alpha_y+\frac{\Phi_{i+1,j}-\Phi_{i,j}}{\Delta x}\right)\rho^n_{i,j}+\frac12\left(-\alpha_y+\frac{\Phi_{i,j+1}-\Phi_{i-1,j+1}}{\Delta x}\right)\rho^n_{i,j+1}.
\end{array}
\right.
\eeq
$\hat h_{\iR,j}$ and $\hat g_{i,\jR}$ are first order monotone numerical fluxes. Similarly, \eqref{eq:high} is
\beq
\label{2Deuler}
\rho^{n+1}_{i, j}=\rho^{n}_{i, j}-\lambda_x \left(\hat H^{rk}_{i+1/2, j}-\hat H^{rk}_{i-1/2, j}\right)- \lambda_y \left(\hat G^{rk}_{i, j+1/2}-\hat G^{rk}_{i, j-1/2}\right).
\eeq
In order to ensure maximum principle, we are looking for type of limiters
\beq
\label{fluxLM}
\left\{
\begin{array}{l}
\tilde{H}_{i+1/2, j}=\theta_{i+1/2, j} \,( \hat H^{rk}_{i+1/2, j} -\hat h_{i+1/2, j})+\hat h_{i+1/2, j}, 
\\
\,
\\
\tilde{G}_{i, j+1/2}=\theta_{i, j+1/2}\, ( \hat G^{rk}_{i, j+1/2} - \hat g_{i, j+1/2})+ \hat g_{i, j+1/2},
\end{array}
\right.
\eeq
such that
\beq
\label{MP}
\rho_m \, \le \, \rho^{n}_{i, j}-\lambda_x \left(\tilde H_{i+1/2, j}-\tilde H_{i-1/2, j}\right)- \lambda_y \left(\tilde G_{i, j+1/2}-\tilde G_{i, j-1/2}\right) \, \le \, \rho_M.
\eeq
(\ref{fluxLM}) and (\ref{MP}) form coupled inequalities for the limiting parameters $\theta_{i+1/2, j}, \theta_{i, j+1/2}$. We will need to find a group of numbers $\Lambda_{L, i, j}, \Lambda_{R, i, j},\Lambda_{D, i, j},\Lambda_{U, i, j}$, such that (\ref{MP}) is satisfied with
\begin{eqnarray*}
	(\theta_{i-1/2, j}, \theta_{i+1/2, j},\theta_{i, j-1/2},\theta_{i, j+1/2}) \, \in \, [0, \Lambda_{L, i, j}]\times [0, \Lambda_{R, i, j}]\times [0, \Lambda_{D, i, j}] \times [0,\Lambda_{U, i, j}].
\end{eqnarray*}
This is achieved in a decoupled way for the minimum and maximum parts.  For the maximum value case, we let
\begin{eqnarray}
\label{1mp}
\Gamma_{i, j} \,=\, \rho_{M}-\left(\rho_{i, j}-\lambda_x (\hat h_{i+1/2, j}-\hat h_{i-1/2, j})- \lambda_y \left(\hat g_{i, j+1/2}-\hat g_{i, j-1/2}\right)\right) \,\ge \, 0,
\end{eqnarray}
%when a monotone numerical flux is used under suitable CFL constraint $\alpha_x\frac{\Delta t}{\Delta x} +\alpha_y\frac{\Delta t}{\Delta y}\le 1$, here $\alpha_x=\max_u|f'(u)|$ and $\alpha_y=\max_u|g'(u)|$.
and denote
\beq
\left\{\begin{array}{l}
F_{i-1/2, j}\,=\,+\lambda_x \left(\hat H^{rk}_{i-1/2, j} -\hat h_{i-1/2, j} \right),
\\ \, \\
F_{i+1/2, j}\,=\, -\lambda_x \left(\hat H^{rk}_{i+1/2, j} -\hat h_{i+1/2, j} \right),
\\ \, \\
F_{i, j-1/2}\,=\,+\lambda_y \left(\hat G^{rk}_{i, j-1/2} -\hat g_{i, j-1/2} \right),
\\ \, \\
F_{i, j+1/2}\,=\,-\lambda_y \left(\hat G^{rk}_{i, j+1/2} -\hat g_{i, j+1/2} \right),
\end{array}
\right.
\eeq
then (\ref{MP}) can be rewritten as
\begin{eqnarray}
\label{cMP}
\theta_{i+1/2, j}  F_{i+1/2, j}+\theta_{i-1/2, j}  F_{i-1/2, j}+\theta_{i, j+1/2}  F_{i, j+1/2}+\theta_{i, j-1/2}  F_{i, j-1/2}\,\le\, \Gamma_{i, j}.
\end{eqnarray}
We shall now focus on decoupling the inequalities (\ref{cMP}): for the single node $(i, j)$,
\begin{enumerate}
	\item Identify positive values out of the four locally defined numbers $F_{i-1/2, j}$, $F_{i+1/2, j}$, $F_{i, j-1/2}$,  $F_{i, j+1/2}$;
	\item Corresponding to those positive values, collectively, the limiting parameters can be defined. For example,
	if $F_{i+1/2, j}, F_{i-1/2, j}>0$ and $F_{i, j-1/2}, F_{i, j+1/2}\le0$, then
    \beq
    \left\{ 
    \begin{array}{l}
    \displaystyle\Lambda^M_{i+1/2, j}= \Lambda^M_{i-1/2, j}=\min (\frac{\Gamma_{i,j}}{F_{i+1/2, j}+F_{i-1/2, j}}, \, 1), 
    \\ \, \\
	\Lambda^M_{i, j-1/2} = \Lambda^M_{i, j+1/2}=1,
	\end{array}
	\right.
	\eeq
\end{enumerate}
that is, all high order fluxes which possibly contribute (beyond that of the first order fluxes, which is not good) to the overshooting or undershooting of the updated value shall be limited by the same scaling.
Similarly we can find $\Lambda^M_{i, j\pm1/2}$ and also similarly for the minimum value case
of $\Lambda^m_{i\pm1/2, j}$ and $\Lambda^m_{i, j\pm1/2}$.
%For the minimum value part, let
%\begin{eqnarray}
%\label{ump}
%\Gamma_{i, j}=u_{m}-(u_{i, j}-\lambda_x (\hat h_{i+1/2, j}-\hat h_{i-1/2, j})- \lambda_y (\hat g_{i, j+1/2}-\hat g_{i, j-1/2})) \le 0.
%\end{eqnarray}
%The coupled inequalities  (\ref{fluxLM}) and (\ref{MP}) can be rewritten as
%\begin{eqnarray}
%\label{cuMP}
%\Gamma_{i, j} \le \theta_{i+1/2, j}  F_{i+1/2, j}+\theta_{i-1/2, j}  F_{i-1/2, j}+\theta_{i, j+1/2}  F_{i, j+1/2}+\theta_{i, j-1/2}  F_{i, j-1/2}.
%\end{eqnarray}
%A similar procedure would be applied
%\begin{enumerate}
%	\item Identify negative values out of the four locally defined numbers $F_{i-1/2, j}$, $F_{i+1/2, j}$, $F_{i, j-1/2}$,  $F_{i, j+1/2}$;
%	\item Corresponding to the negative values, {\bf collectively}, the limiting parameters can be defined. For example, if $F_{i, j-1/2}, F_{i, j+1/2}\ge 0$ and $F_{i-1/2, j}, F_{i+1/2, j}<0$, then
%	\begin{eqnarray}
%	\label{LC}
%	\begin{cases} \Lambda^m_{i-1/2, j}, \Lambda^m_{i+1/2, j}=\min(\frac{\Gamma_{i,j}}{F_{i-1/2, j}+F_{i+1/2, j}}, 1)\\
%	\Lambda^m_{i, j-1/2}, \Lambda^m_{i, j+1/2}=1.
%	\end{cases}
%	\end{eqnarray}
%\end{enumerate}
%Namely, all high order fluxes which possibly contribute (beyond that of the first order fluxes) to the overshooting or undershooting of the updated value shall be limited by the same scaling. Similarly we can find $\Lambda^M_{i, j\pm1/2}$ and $\Lambda^m_{i, j\pm1/2}$, 
Therefore the range of the limiting parameters satisfying MPP for a single node $(i, j)$ is defined by
\beq
\left\{
\begin{array}{l}
\Lambda_{L, i, j} \,=\, \min \, (\Lambda^M_{i-1/2, j}, \, \Lambda^m_{i-1/2, j}),  
\\ \, \\
\Lambda_{R, i, j} \,=\, \min \, (\Lambda^M_{i+1/2, j}, \, \Lambda^m_{i+1/2, j}),
\\ \, \\ 
\Lambda_{U, i, j} \,=\, \min \, (\Lambda^M_{i, j+1/2}, \, \Lambda^m_{i, j+1/2}),
\\ \, \\
\Lambda_{D, i, j} \,=\, \min \, (\Lambda^M_{i, j-1/2}, \, \Lambda^m_{i, j-1/2}).
\end{array}
\right.
\eeq
Considering the limiters from neighboring nodes, finally we let
\beq
\label{eq:theta}
\left\{
\begin{array}{l}
\theta_{i+1/2, j}=\min\,(\Lambda_{R, i, j}, \, \Lambda_{L, i+1, j}),
\\ \, \\
\theta_{i, j+1/2}=\min\,(\Lambda_{U, i, j} , \, \Lambda_{D, i, j+1}).
\end{array}
\right.
\eeq
Substituting \eqref{eq:theta} into \eqref{fluxLM}, our final scheme is \eqref{2Deuler} by replacing
the fluxes from \eqref{fluxLM}.

%%%%%%%%%%%%%%%%%%%
%
%%%%%%%%%%%%%%%%%%%
\section{Numerical examples}
\label{sec5}
\setcounter{equation}{0}
\setcounter{figure}{0}
\setcounter{table}{0}

In this section, we will apply the 5th order finite difference scheme with Hermite linear reconstruction \eqref{eq:linear5} and \eqref{eq:linear5G}, coupled with the bound preserving MPP flux limiter, which we denote as ``HLinear5 MPP", to test its good performance. The Hermite linear scheme without bound preserving MPP flux limiter is denoted as ``HLinear5", while those with Hermite WENO reconstruction \eqref{eq:weno5} are denoted as ``HWENO5 MPP" and ``HWENO5" for with and without MPP flux limiter correspondingly. All the schemes without Hermite reconstruction, for example the one in \cite{xiong2013parametrized}, will be denoted as ``WENO5 MPP" and ``WENO5", and the corresponding linear schemes as ``Linear5 MPP" and ``Linear5", respectively. They will also be used for comparison in the following. The time step is taken as
\[
\Delta t \,=\, \text{CFL} /(\alpha_x/\Delta x + \alpha_y/\Delta y),
\]
where $\alpha_x = \|U_x\|_{L^\infty}$, $\alpha_y=\|U_y\|_{L^\infty}$ and we take the CFL number to be $\text{CFL} = 0.6$ in all the following tests.

We would emphasize that to apply the MPP limiter, the minimum and maximum values of the bounds are taken from the initial data, which can be explicitly determined for all the tests here.

\subsection{Linear transport equations}

\begin{exa}
	\label{ex1}
	We first take a 2D linear transport equation $\rho_t+\rho_x+\rho_y=0$ with initial condition
	\beq
	\rho(0,x,y)=\sin^4(x)+\sin^4(y),
	\eeq 
	on the domain $[0,\,2\pi]^2$ with periodic boundary conditions, to test the accuracy of the ``HLinear5 MPP" and ``HLinear5" schemes. The exact solution is 
	\beq
	\rho(t,x,y)=\sin^4(x-t)+\sin^4(y-t).
	\eeq
	The solution is within $[0,\,2]$. From Table \ref{tab1},
	we can clearly observe undershootings without the MPP limiter, while with limiter the lower bound is well preserved (as is for the upper bound) and 5th order accuracy is maintained. We would mention
	that for all other schemes we considered in this section, the accuracies are all tested and perform similarly. We omit them to save space.
	
	\begin{table}[h]
		\centering
		\caption{Accuracy test for the 2D linear transport equation in Example \ref{ex1}.  $\rho_{min}$ is the minimum numerical solution. ``WO'' stands for ``without limiter'', ``WL'' stands for ``with limiter''. $T=1$.}
		\vspace{0.2cm}
		\begin{tabular}{|c|c|c|c|c|c|c|}
			\hline
			& N  & $L^1$ error &    order   & $L^\infty$ error & order & $\rho_{min}$ \\\hline
			\multirow{5}{*}{WO}
			&  32 & 5.00e-04 &  --   & 1.29e-03  &  --   &  -0.001019   \\ \cline{2-7}
			&  64 & 1.90e-05 & 4.72 & 4.91e-05 & 4.72 &  -4.244e-05  \\ \cline{2-7}
			& 128 & 6.41e-07 & 4.89 & 1.68e-06 & 4.87 &  -1.242e-06  \\ \cline{2-7}
			& 256 & 2.05e-08 & 4.97  & 5.35e-08 & 4.97 &  -4.403e-08  \\ \hline
			\multirow{5}{*}{WL}
			&  32 & 5.16e-04 &  --   & 1.33e-03  &  --   &  0.0003958   \\ \cline{2-7}
			&  64 & 1.90e-05 & 4.77 & 5.54e-05 & 4.59 &  5.522e-06   \\ \cline{2-7}
			& 128 & 6.41e-07 & 4.89 & 2.07e-06 & 4.74 &  9.832e-13   \\ \cline{2-7}
			& 256 & 2.05e-08 & 4.97  & 7.78e-08  & 4.73 &  1.712e-09   \\ \hline
		\end{tabular}
		\label{tab1}
	\end{table}
	
	Then we take a 1D linear transport equation $\rho_t + \rho_x = 0$ with highly oscillatory solutions.
	Here the initial condition is chosen to be
	\beq
	\rho(0,x) = \sin(4x(x-2\pi)),
	\eeq
	on the domain $[0,\,2\pi]$ with periodic boundary condition, and the exact solution is
	\beq
	\rho(t,x) = \sin(4(x-t)(x-2\pi-t)).
	\label{exs12}
	\eeq
	
	In Fig. \ref{fig2}, we show the results for both linear and WENO type schemes, without and
	with MPP limiter respectively. From the figures, we can observe that the ``HLinear5" scheme has better resolution than the ``Linear5" scheme, due to a Hermite reconstruction for such a highly oscillatory solutions, while linear schemes are better than WENO type schemes due to less dissipation. We can also see that the resolution of these schemes do not significantly affect by the MPP limiter, which is used to control the solutions within the global physical range $[-1,\,1]$. So in the following, we will focus on the ``HLinear5" and ``HLinear5 MPP" schemes, and compare them
	to their corresponding WENO type approaches, which are ``HWENO5" and ``HWENO5 MPP" respectively. We mainly would like to show that a 5th order linear scheme with a Hermite reconstruction and an MPP limiter, that is ``HLinear5 MPP", is a very good scheme for highly oscillatory solution without discontinuities.
	
	%In Fig. \ref{fig1}, we present the solutions at $T=1.5$ with grid $N=160$ for the numerical solutions, while the exact solution is with grid $N=1800$.  First without limiter, we compare the two schemes
	%with linear weights (Top left, resulting in linear schemes respectively) to clearly show the different performance of them . We can find that our HWENO5 scheme can capture the peaks of the highly oscillating solution structures much better than the WENO5 scheme, especially on the interval of $[0, \, 2]$. 
	%Then we compare the two schemes both with nonlinear weights (Top right), we can see that
	%nonlinear weights would smear the numerical solutions on the highly oscillating part, however, we still can find that our HWENO5 scheme has better resolutions than the classic WENO5 scheme. On the bottom of Fig. \ref{fig1}, we show the results with limiter corresponding to the top figures respectively, the limiter almost does not affect the resolution of the original schemes.
	
	%\begin{center}
	%\begin{figure}
	%\begin{tabular}{cc}
	%\includegraphics[height=2.5in,width=3.0in]{./pic/TRS/u1.png} &
	%\includegraphics[height=2.5in,width=3.0in]{./pic/TRS/u2.png}\\
	%(a) & (b) \\
	%\includegraphics[height=2.5in,width=3.0in]{./pic/TRS/u1mpp.png} &
	%\includegraphics[height=2.5in,width=3.0in]{./pic/TRS/u2mpp.png} \\
	%(c) & (d)
	%\end{tabular}
	%\caption{1D linear equation with exact solution \eqref{exs12} at $T=1.5$. Left: linear weights; right: nonlinear weights. Top: without limiter; bottom: with limiter.}
	%\label{fig1}
	%\end{figure}
	%\end{center}
	
	\begin{center}
		\begin{figure}
			\begin{tabular}{cc}
				\includegraphics[height=2.5in,width=3.0in]{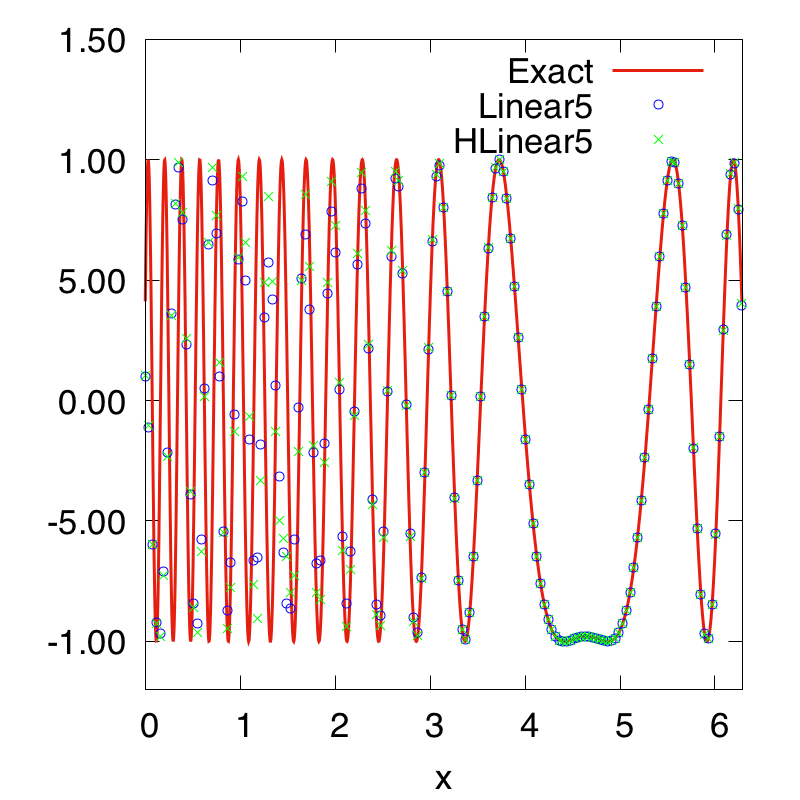} &
				\includegraphics[height=2.5in,width=3.0in]{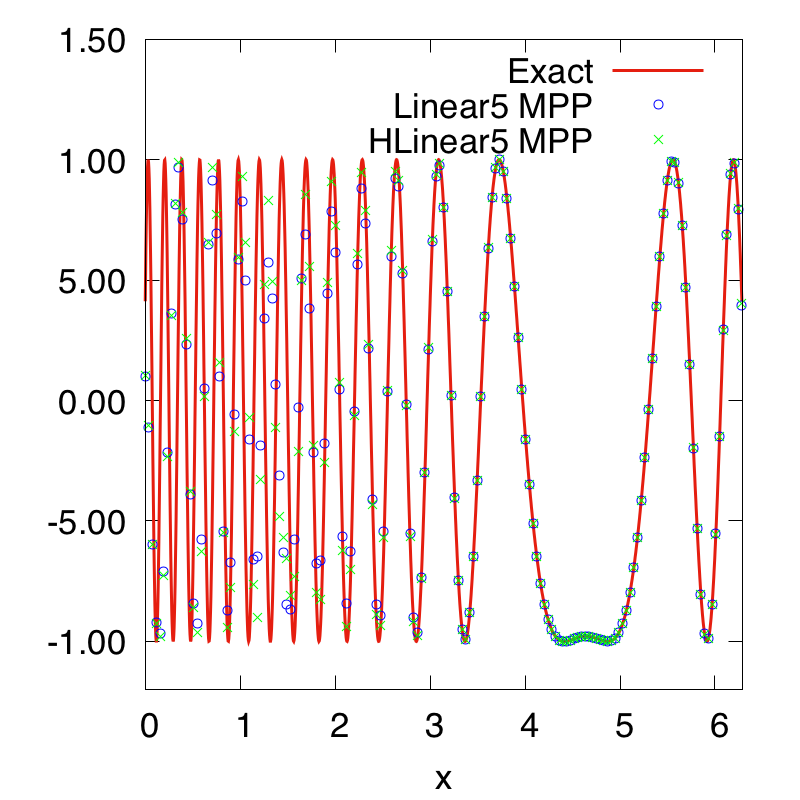}\\
				(a) & (b) \\
						\includegraphics[height=2.5in,width=3.0in]{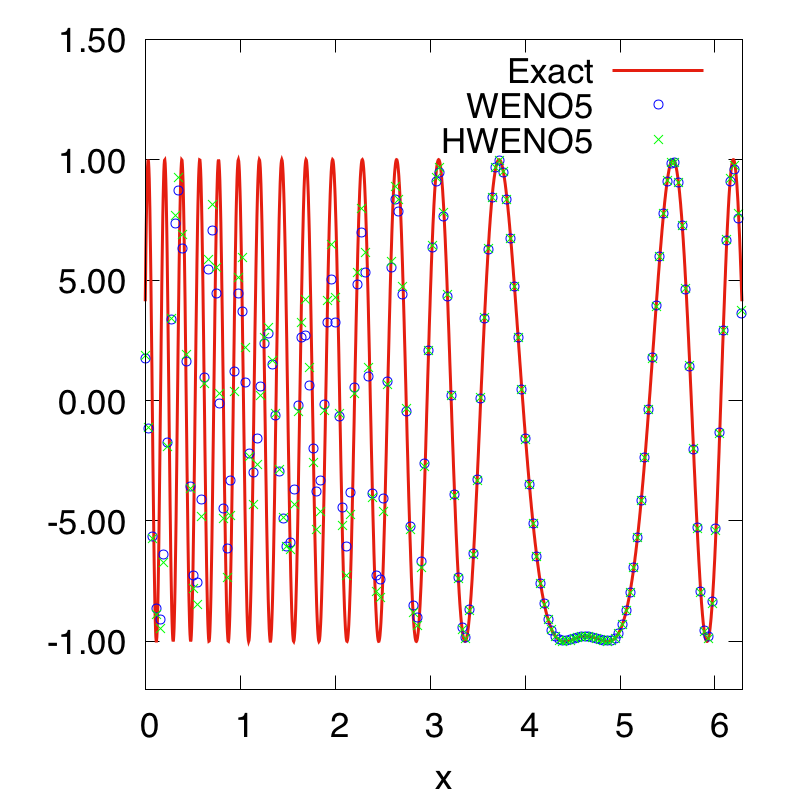} &
						\includegraphics[height=2.5in,width=3.0in]{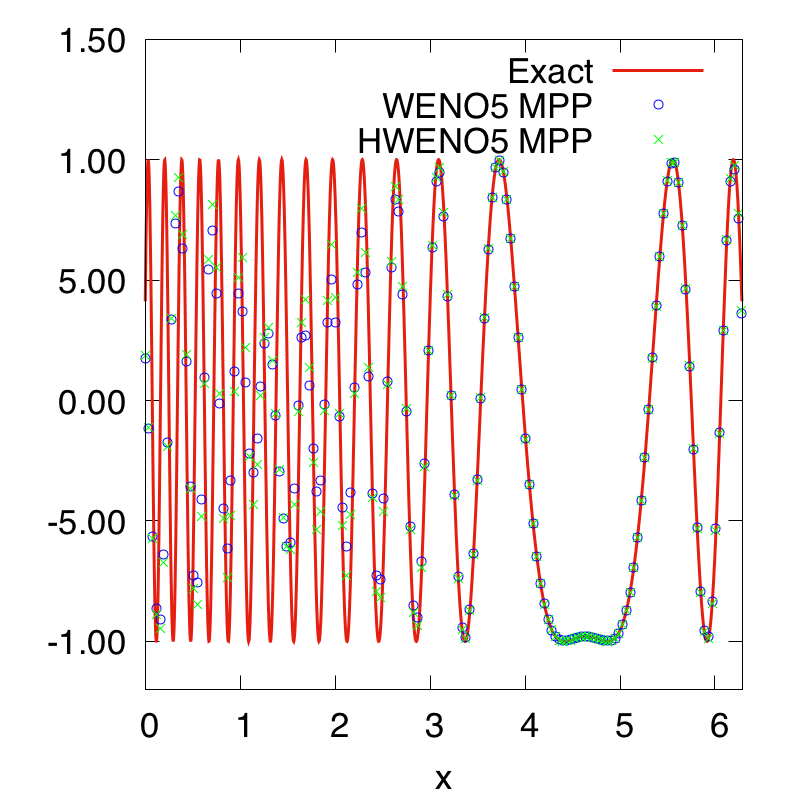}\\
							(c) & (d) \\
			\end{tabular}
			\caption{1D linear equation with exact solution \eqref{exs12} at $T=1.5$. Top: linear schemes; bottom: WENO schemes; left: without MPP limiter; right: with MPP limiter. $N=160$.}
			\label{fig2}
		\end{figure}
	\end{center}
	
\end{exa}

\subsection{Vlasov-Poisson system}
%The Vlasov-Poisson (VP) system is arising from the collisionless plasma physics,
%which reads
%\begin{equation}
%\frac{\partial f}{\partial t} + {\bf v} \cdot \nabla_{\bf x} f  +
%\mathbf{E}({\bf x},t) \cdot \nabla_{\bf v} f = 0, \label{eq: vlasov}
%\end{equation}
%and
%\begin{equation}
%\mathbf{E}(\mathbf{x},t)=-\nabla_{\bf x}\Phi(\mathbf{x},t),\quad
%-\Delta_{\bf
%x}\Phi(\mathbf{x},t)=\rho(\mathbf{x},t)-1.\label{eq: poisson3}
%\end{equation}
%It describes the evolution of the particle distribution function
%in a 6 dimensional phase space. The distribution function $f( {\bf x},{\bf v},t)$
%describes the probability of finding a particle with velocity $\bf{v}$ at position $\bf{x}$ at
%time $t$, $\bf{E}$ is the electric field, and $\Phi$ is the
%self-consistent electrostatic potential. The probability
%distribution function couples to the long range fields via the
%charge density, $\rho(t,x) = \int_{\mathbb{R}^3} f(x,v,t)dv$,
%where we take the limit of uniformly distributed infinitely massive
%ions in the background. 

In this subsection, we consider the VP system \eqref{eq:vlasov} and \eqref{eq:poisson}  with 1D in ${\bf x}$ and 1D in ${\bf v}$, and periodic boundary condition in both directions. We take $\bfx = (x,v)$ and use $N_v$ instead of $N_y$ for the special meaning of the argument $v$ here. The cut-off domain in the $v$ direction is taken to be $[-2\pi,2\pi]$ if without specifications. The VP system preserves several norms, which should remain constant in time:
\begin{enumerate}
%\item Mass:
%\[
%\text{Mass}=\int_v\int_xf(x,v,t)dxdv.
%\]
\item $L^p$ norm, $1\leq p<\infty$:
\begin{equation}
\|f\|_p=\left(\int_v\int_x|f(x,v,t)|^pdxdv\right)^\frac1p.
\end{equation}
\item Energy:
\begin{equation}
\text{Energy}=\int_v\int_xf(x,v,t)v^2dxdv + \int_xE^2(x,t)dx,
\end{equation}
where $E(x,t)$ is the electric field.
\item Entropy:
\begin{equation}
\text{Entropy}=\int_v\int_xf(x,v,t)\log(f(x,v,t))dxdv.
\end{equation}
\end{enumerate}
We will track the relative deviations of these quantities numerically to measure the quality of our numerical scheme.
The relative deviation is defined to be the deviation away from the corresponding initial value divided by the magnitude
of the initial value. The mass conservation over time $\int_v\int_x f(x,v,t)dxdv$ is obvious for  conservative schemes, which is the same as the $L^1$ norm when $f$ is positive, so only the $L^1$ norm is shown below. We would also note that for the VP system, although the minimum part is known as positivity preserving, however, due to a cut-off domain in the $v$ direction,
the minimum value might be close to but above $0$. We clearly indicate the minimum and maximum values from the initial data for the tests below.

%\begin{exa} (Weak Landau damping)
%For the VP system, we first consider the weak Landau  damping with the initial condition:
%\beq
%\label{landau}
%f(t=0,x,v)=\f{1}{\sqrt{2\pi}}(1+\alpha \cos(k x))\exp(-\f{v^2}{2}),
%\eeq
%where $\alpha=0.01$ and $k=0.5$. The length of the domain in the x-direction is
%$L=\f{2\pi}{k}$, which is similar in the following examples.
% In Fig. \ref{fig21}, we plot the time evolution of the electric field in $L^2$ norm and $L^\infty$ norm, the relative derivation of the discrete $L^1$ norm, $L^2$ norm, kinetic energy and entropy.
%
%\begin{figure}
%\centering
%\includegraphics[width=3in,clip]{./pic/vp/wl_E.png},
%\includegraphics[width=3in,clip]{./pic/vp/wl_en.png} \\
%\includegraphics[width=3in,clip]{./pic/vp/wl_l1.png},
%\includegraphics[width=3in,clip]{./pic/vp/wl_l2.png}
%\caption{Weak Landau damping. Time evolution of the electric field in $L^2$ norm and $L^\infty$ norm, kinetic energy and entropy, discrete $L^1$ norm and $L^2$ norm. Mesh: $128 \times 128$.}
%\label{fig21}
%\end{figure}
%\end{exa}

\begin{exa}\label{ex61}(Accuracy test)
	We first consider the VP system with the following initial condition
	\beq
	f(0,x,v)=\f{1}{\sqrt{2\pi}}\cos^4(k x)\exp(-\f{v^2}{2}),
	\label{vpsmooth}
	\eeq
	and periodic boundary conditions on the computational domain
	$[0, 4\pi]\times[-4\pi,4\pi]$, where $k=0.5$, to test the accuracy of the schemes
	for this system. From the initial data, the solution should be within the range $[0,\, \f{1}{\sqrt{2\pi}}]$.
	
	In Table \ref{tab3}, we show the $L^1$ and $L^{\infty}$ errors and orders for the  ``HLinear5" and ``HLinear5 MPP" scheme respectively. For this example, $5$th order accuracy can also be observed as mesh refinement. Without limiter, the solution of the distribution function does not preserve positivity, while with limiter, all values are above $0$.
	Here we measure the errors on two consecutive mesh sizes by double refinement, since the exact solution is not explicitly available.
	
	\begin{table}
		\centering
		\caption{$L^1$ and $L^{\infty}$ errors and orders for the VP system with initial condition (\ref{vpsmooth}) at $T=1$. ``WL" denotes the scheme with limiters, ``WO" denotes the scheme without limiters. ``$f_{min}$" is the minimum of the numerical solution. Mesh size $N_v=2N_x$.}
		\vspace{0.2cm}
		\begin{tabular}{|c|c|c|c|c|c|c|}
			\hline
			& $N_x$  & $L^1$ error &    order   & $L^\infty$ error & order & $f_{min}$ \\ \hline
			\multirow{5}{*}{WO}
           &64 & 1.66e-05 &  --  & 0.0004144 &  --  &  -1.317e-06 \\ \cline{2-7}
           &128 & 6.83e-07 & 4.60 & 1.913e-05 & 4.44 & -2.197e-08 \\ \cline{2-7}
           &256 & 2.54e-08 & 4.75 & 7.137e-07 & 4.75 &  -1.777e-09 \\ \cline{2-7}
           &512 & 8.29e-10 & 4.94 & 2.342e-08 & 4.93 &  -5.535e-11 \\ \hline
			\multirow{5}{*}{WL}
           &64 & 1.68e-05 &  --  & 0.0004177 &  --  &  1.59e-32 \\ \cline{2-7}
           &128 & 6.85e-07 & 4.62 & 1.915e-05 & 4.45 & 1.768e-34 \\ \cline{2-7}
          &256 & 2.54e-08 & 4.75 & 7.136e-07 & 4.75 & 2.339e-35 \\ \cline{2-7}
         &512 & 8.30e-10 & 4.94 & 2.342e-08 & 4.93 & 4.797e-36 \\ \hline
		\end{tabular}
		\label{tab3}
	\end{table}
	
\end{exa}

\begin{exa}\label{ex62}(Strong Landau damping)
We then consider the strong Landau damping with the initial condition:
\beq
\label{landau}
f(0,x,v)=\f{1}{\sqrt{2\pi}}(1+\alpha \cos(k x))\exp(-\f{v^2}{2}),
\eeq
where $\alpha=0.5$ and $k=0.5$. The length of the domain in the x-direction is
$L=\f{2\pi}{k}$, which is similar in the following two examples. For this problem,
from the initial data, the solution should be within the range $[\f{1}{\sqrt{2\pi}}(1-\alpha)\exp(-\f{(2\pi)^2}{2}), \, \f{1}{\sqrt{2\pi}}(1+\alpha)]$. 
In Fig. \ref{fig20}, we plot
the surface of the distribution function $f$ at $T=50$ in the range of $[0, 0.6]$. The mesh grid is $256\times256$. We can observe that without limiter, the solution can be negative (white spots), while negative values are eliminated by the scheme with limiter, similarly for the following examples. Then in Fig. \ref{fig21}, we show the time evolution of the electric field in $L^2$ norm and $L^\infty$ norm, the relative derivation of the discrete $L^1$ norm, $L^2$ norm, kinetic energy and entropy. From this figure,
we can see that the electric field for all schemes are almost the same. The linear type scheme can preserve the $L^2$ norm and entropy better than the WENO type scheme, where the MPP limiter does 
not significantly affect these two quantities. However, for the $L^1$ norm and energy, the linear scheme
without limiter performs much worse than the WENO type scheme, but after with limiter, it is the best.
Especially for the energy, ``HLinear5 MPP" is much better than all other schemes. This example has clearly show the good performance of our approach.

\begin{figure}
\centering
\includegraphics[width=3.2in]{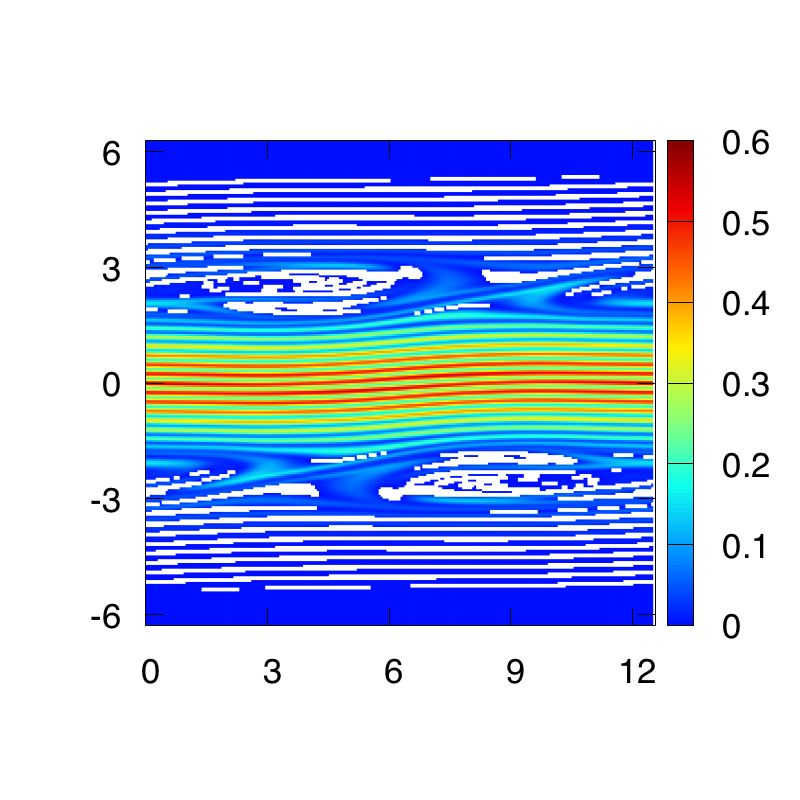}
\includegraphics[width=3.2in]{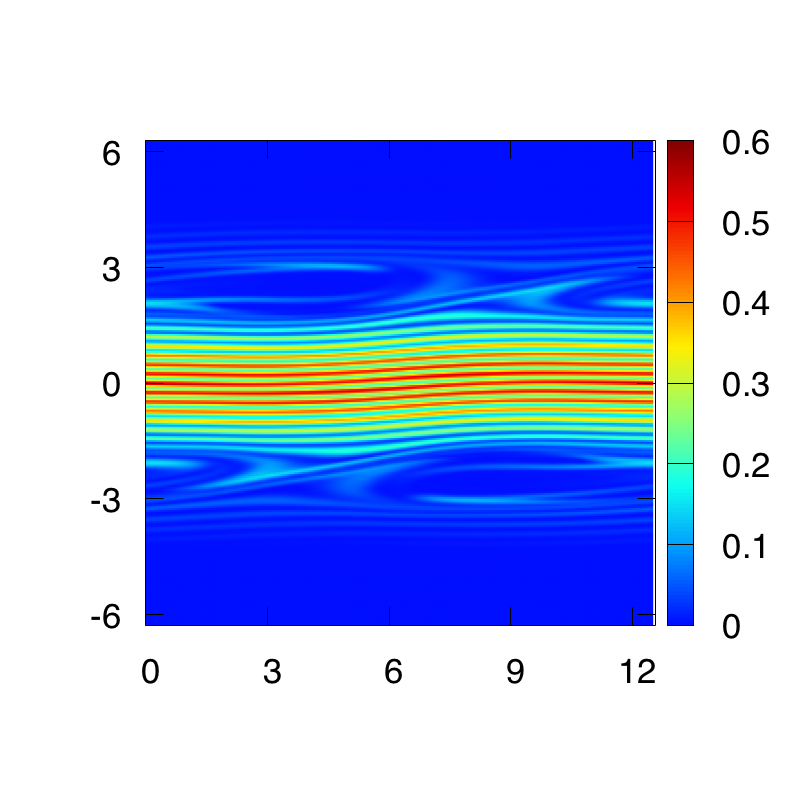}
\caption{Strong Landau damping at $T=50$. Left: without limiter; Right: with limiter. Mesh grid: $256 \times 256$. }
\label{fig20}
\end{figure}

\begin{figure}
\centering
\includegraphics[width=3in,clip]{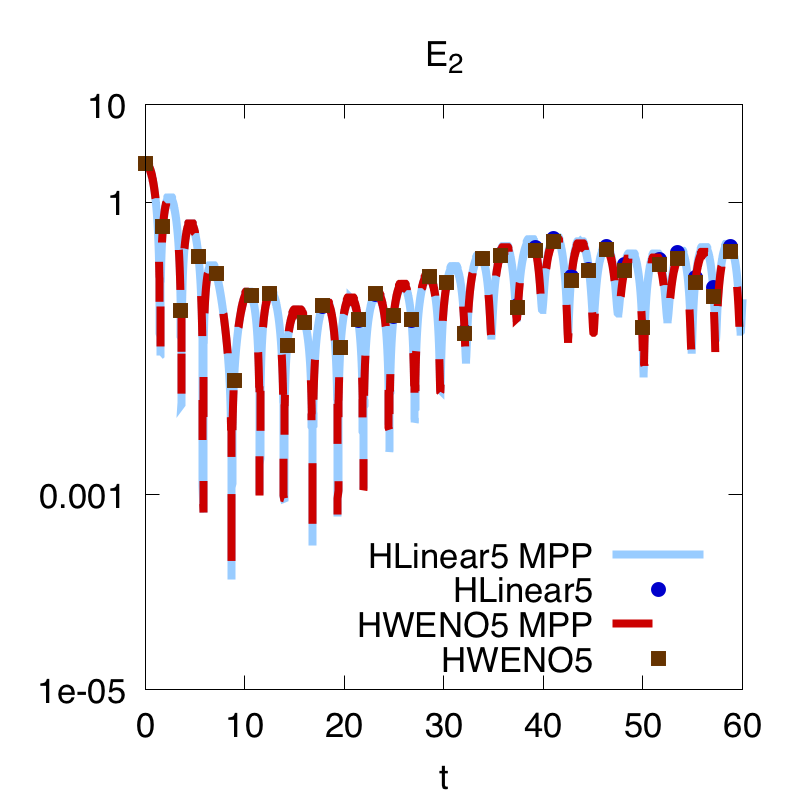}
\includegraphics[width=3in,clip]{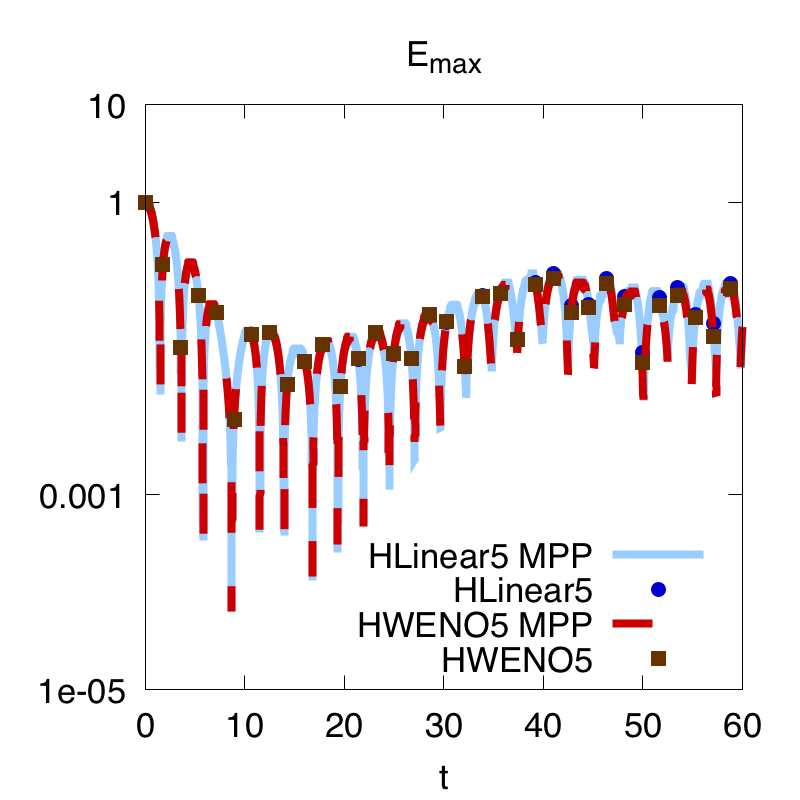}\\
\includegraphics[width=3in,clip]{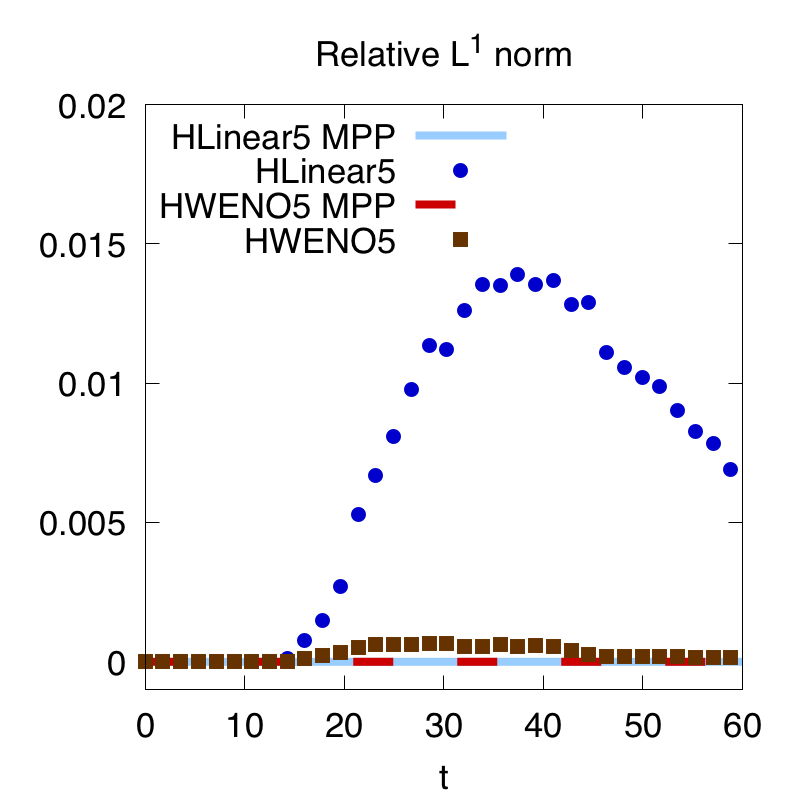}
\includegraphics[width=3in,clip]{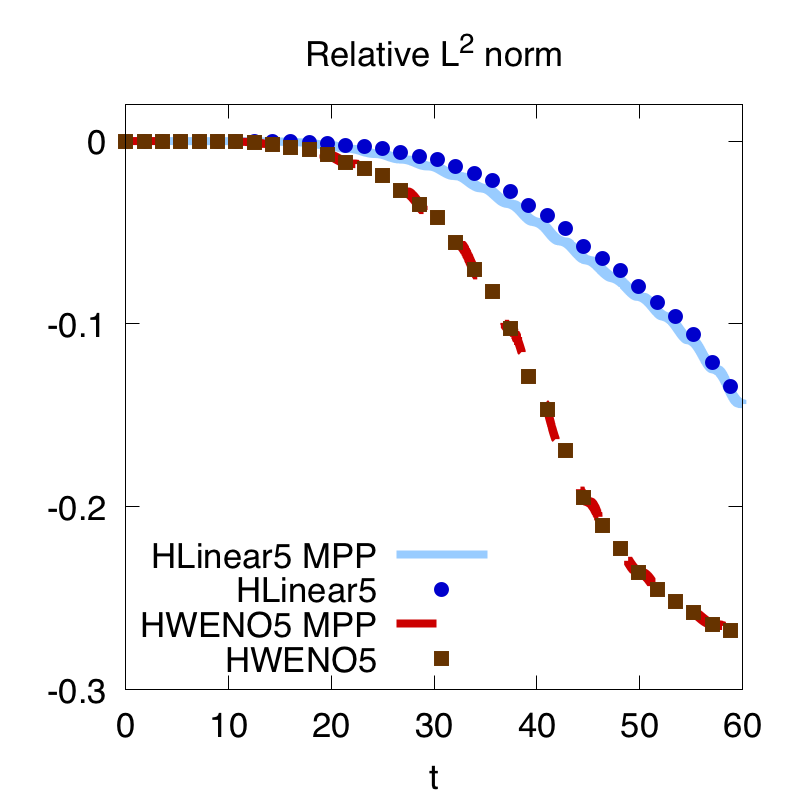}\\
\includegraphics[width=3in,clip]{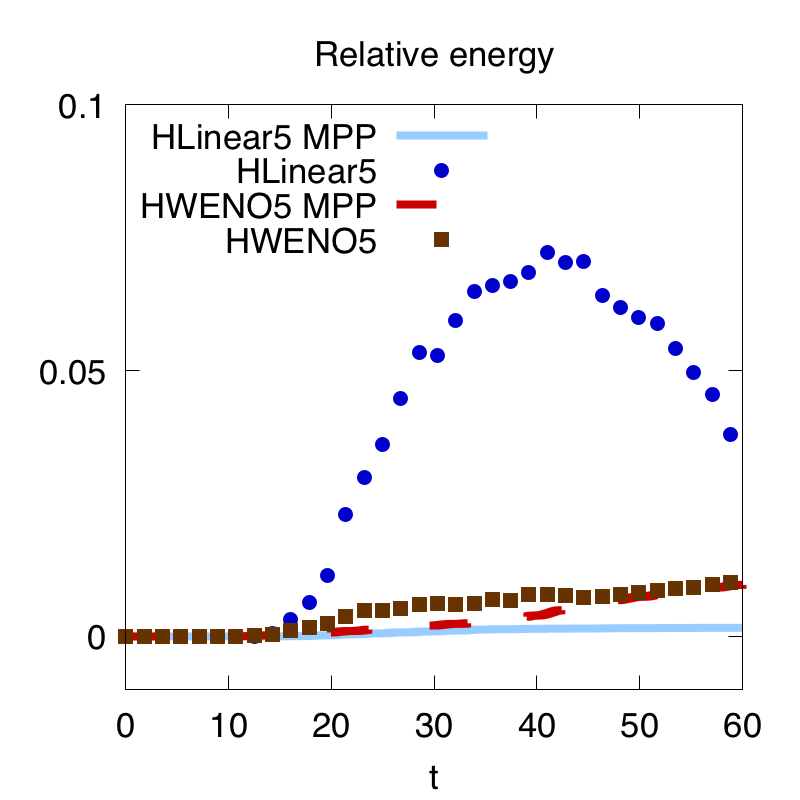}
\includegraphics[width=3in,clip]{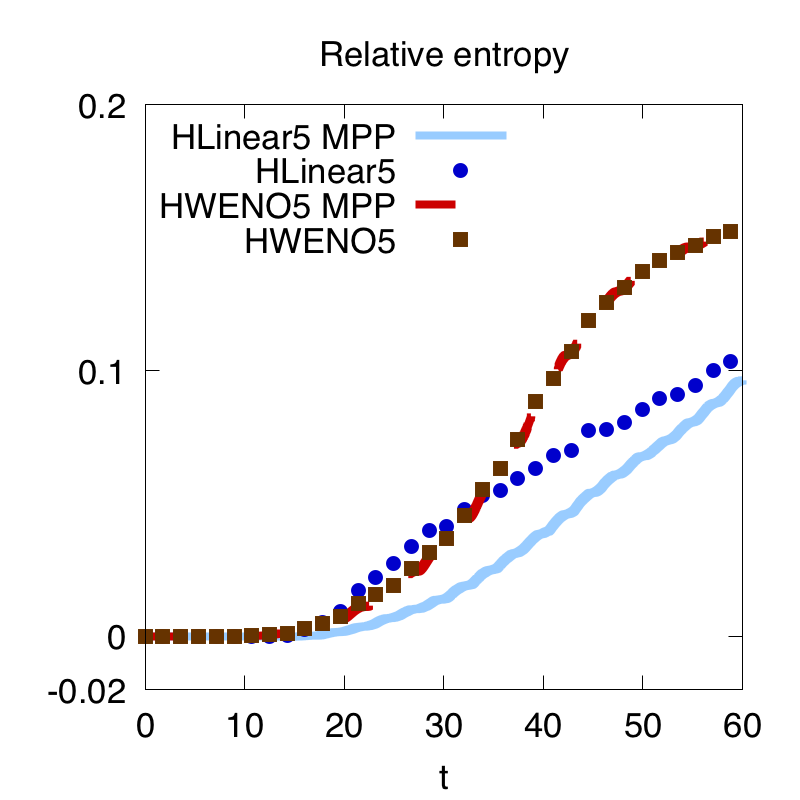} 
\caption{Strong Landau damping. Time evolution of the electric field in $L^2$ norm and $L^\infty$ norm (top), relative differences of discrete $L^1$ norm and $L^2$ norm (middle), relative differences of kinetic energy and entropy (bottom). Mesh grid: $256 \times 256$.}
\label{fig21}
\end{figure}

\end{exa}

\begin{exa} \label{ex63} (Symmetric two stream instability)
We now consider the symmetric two stream instability with the initial condition:
\beq
\label{s2s}
f(0,x,v)=\f{1}{2v_{th}\sqrt{2\pi}}\left[\exp\left(-\f{(v-u)^2}{2v_{th}^2}\right)
+\exp\left(-\f{(v+u)^2}{2v_{th}^2}\right)\right](1+\alpha\cos(kx))
\eeq
where $\alpha=0.05$, $u=0.99$, $v_{th}=0.3$ and $k=\f{2}{13}$. Similarly, from the initial data,
the minimum value of the solution is taking $v=2\pi$ and $\cos(kx)=-1$ in \eqref{s2s}, while the maximum value is taking $v=u$ or $v=-u$, and $\cos(kx)=1$.
We plot the numerical solution at $T=70$ in Fig. \ref{fig24} for the ``HLinear5" and ``HLinear5 MPP" schemes respectively. The mesh grid is $256\times256$. We can also clearly observe that without limiter, the solution becomes negative, which, however, can be eliminated by the scheme with limiter.  We show the time evolution of the electric field in $L^2$ norm and $L^\infty$ norm, the relative derivation of the discrete $L^1$ norm, $L^2$ norm, kinetic energy and entropy in Fig. \ref{fig25}. Similar
results as the last example are obtained.

\begin{figure}
\centering
\includegraphics[width=3.2in]{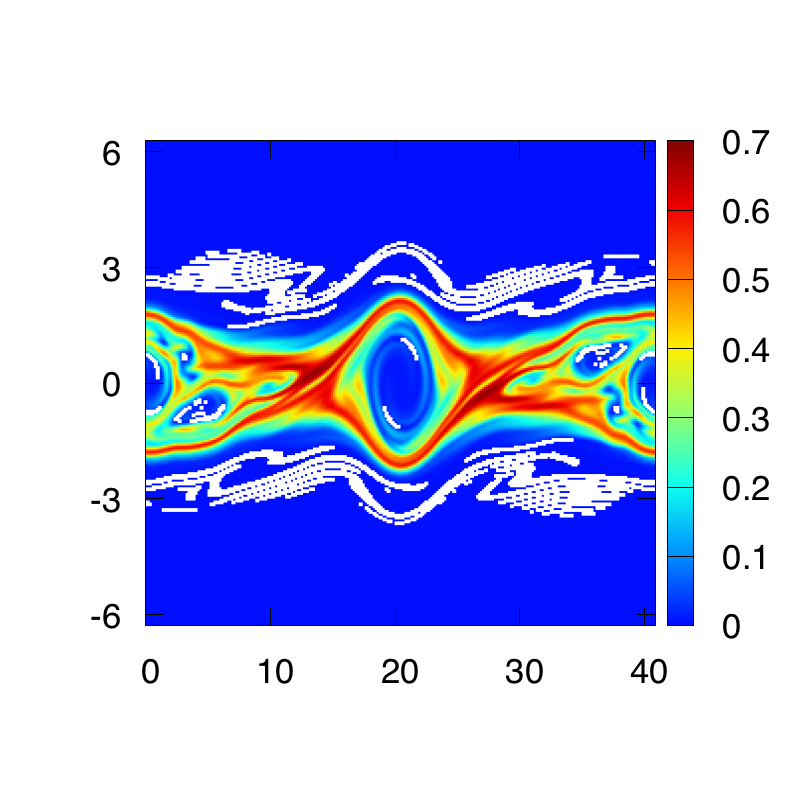}
\includegraphics[width=3.2in]{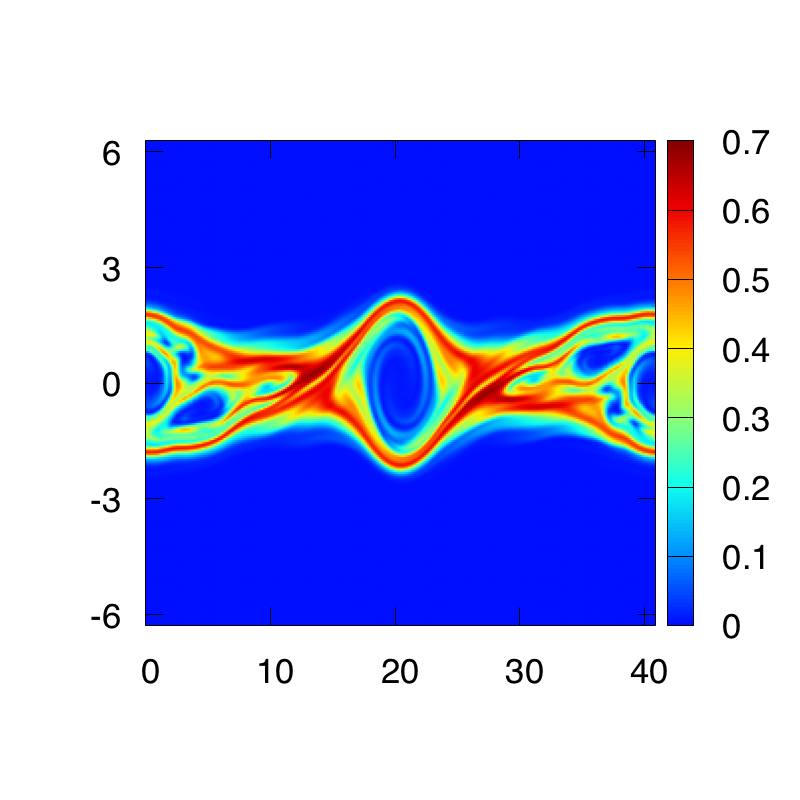}
\caption{Symmetric two stream instability at $T=70$.
Left: without limiter; Right: with limiter. Mesh grid: $256 \times 256$. }
\label{fig24}
\end{figure}

\begin{figure}
\centering
\includegraphics[width=3in,clip]{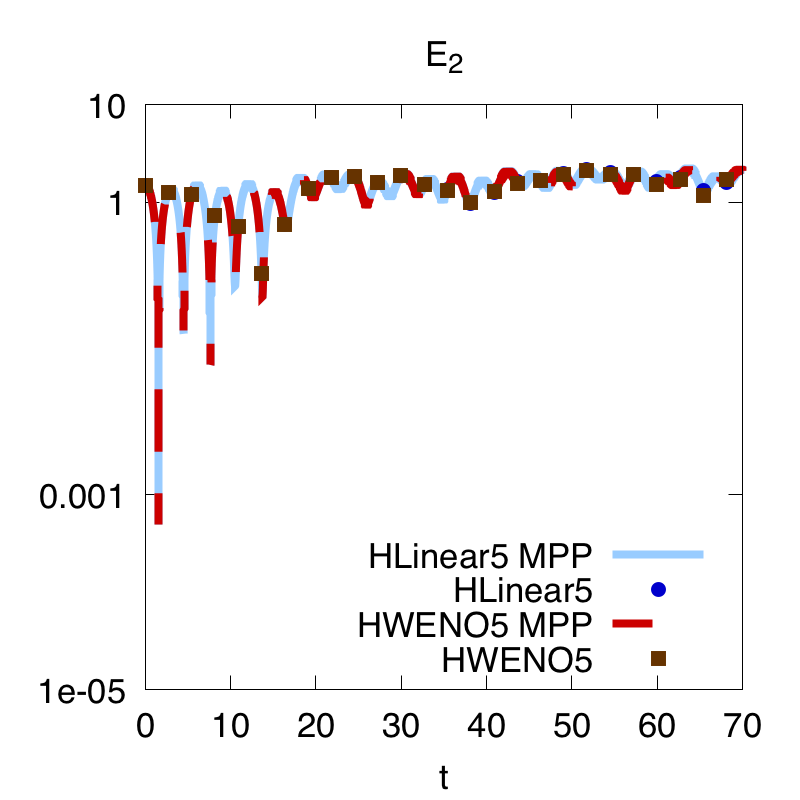}
\includegraphics[width=3in,clip]{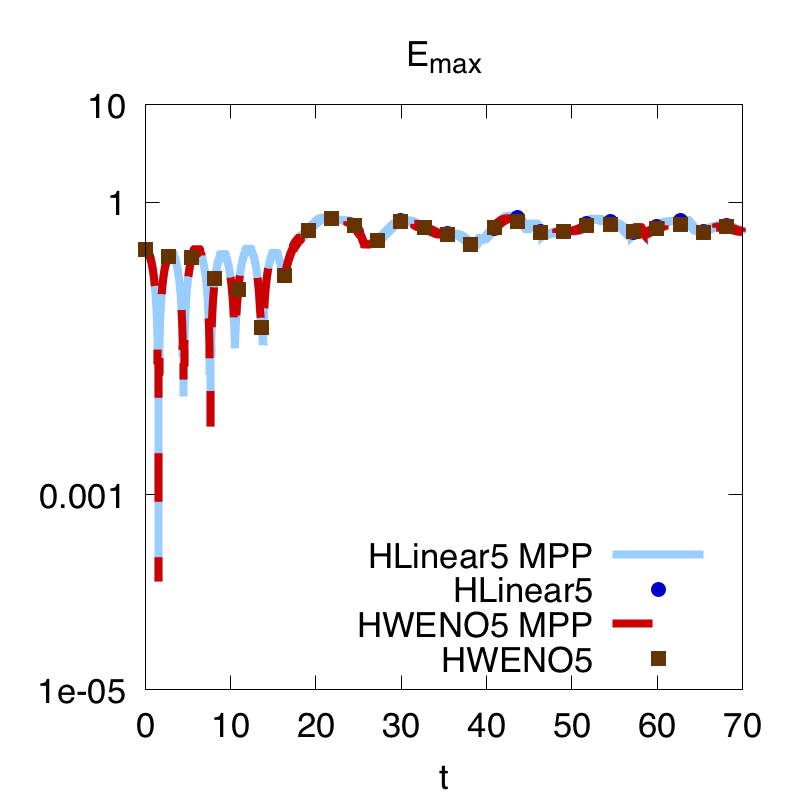}\\
\includegraphics[width=3in,clip]{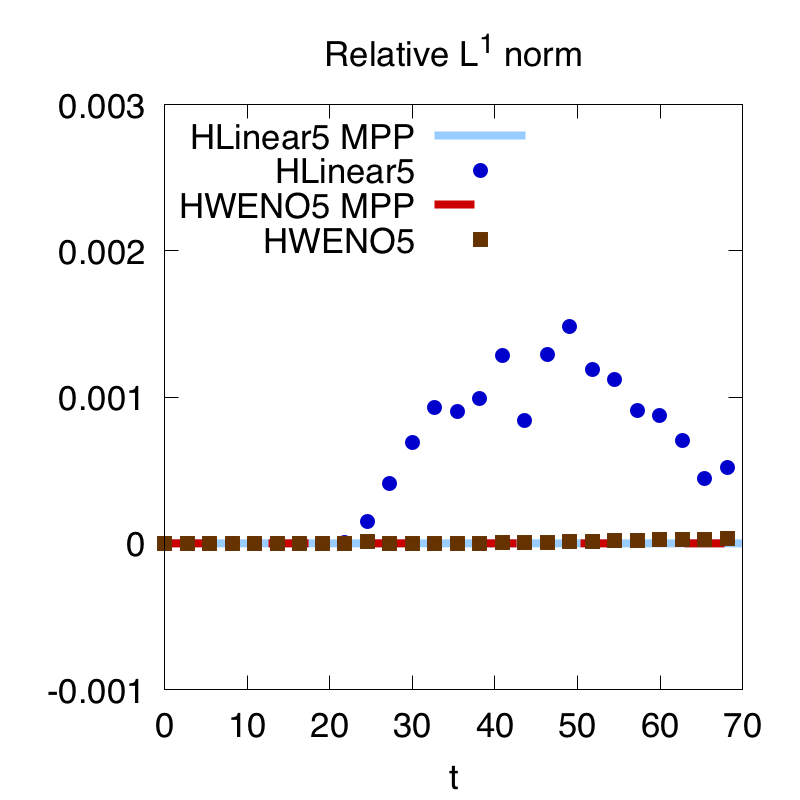}
\includegraphics[width=3in,clip]{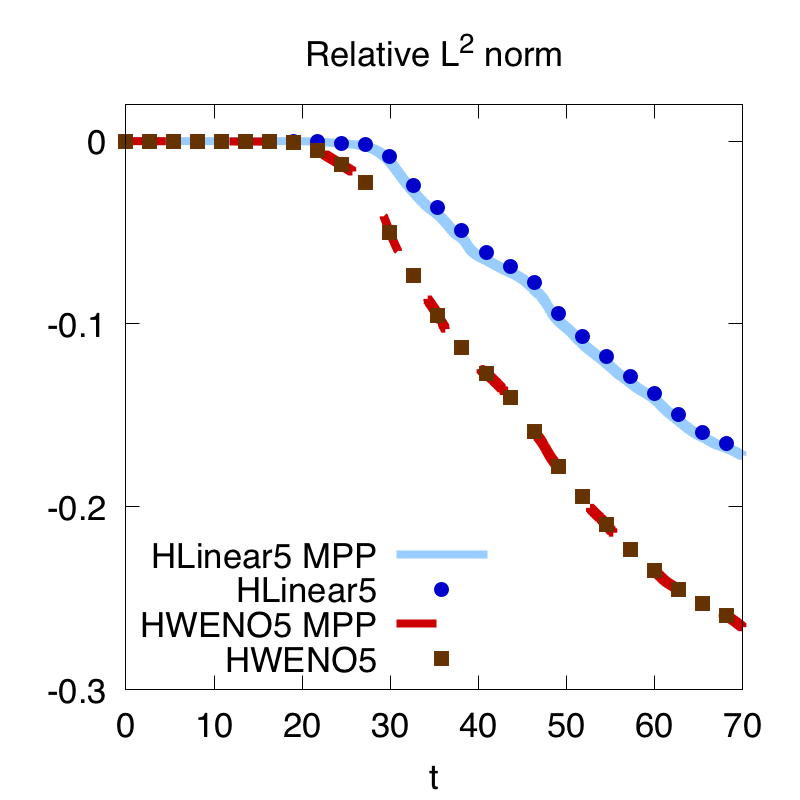}\\
\includegraphics[width=3in,clip]{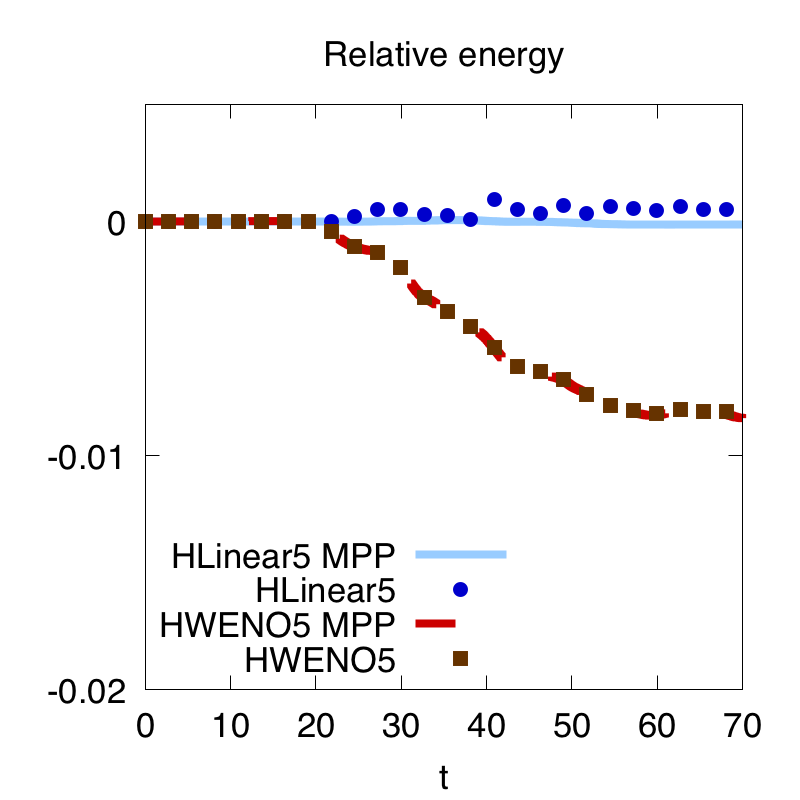}
\includegraphics[width=3in,clip]{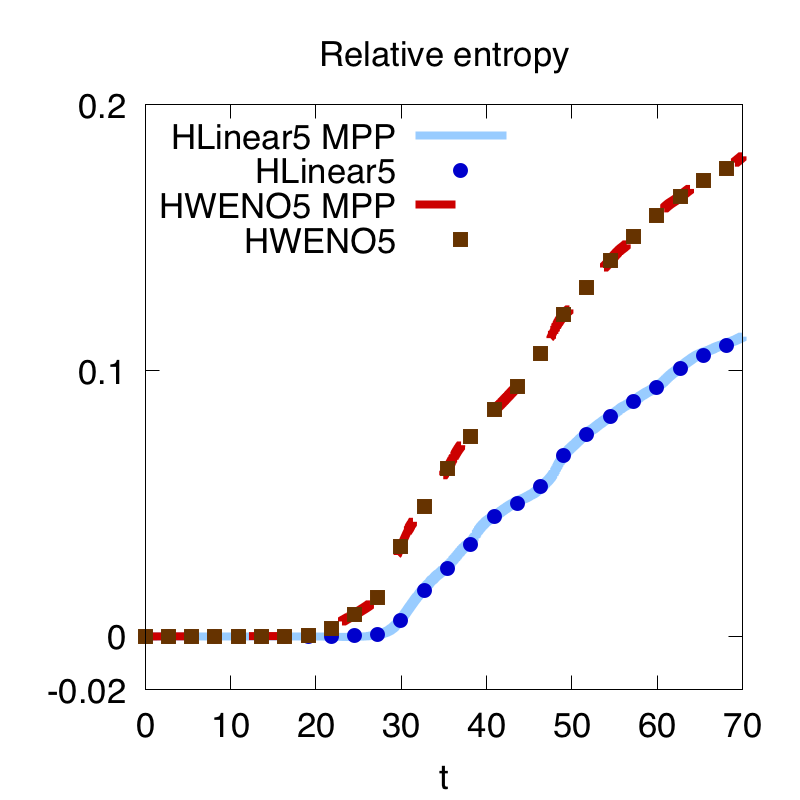}
\caption{Symmetric two stream instability. Time evolution of the electric field in $L^2$ norm and $L^\infty$ norm (top), relative differences of discrete $L^1$ norm and $L^2$ norm (middle), relative differences of kinetic energy and entropy (bottom). Mesh grid: $256 \times 256$.}
\label{fig25}
\end{figure}

\end{exa}

\begin{exa}\label{ex64}(Bump-on-tail instability)
We consider the bump-on-tail instability problem \cite{arber2002critical,xiong2014high} with the initial condition as
	\begin{align}
	\label{bump}
	f(0,x,v)=f_{b.o.t}(v)(1+\alpha\cos(kx)),
	\end{align} 
	where the bump-on-tail distribution is
	\begin{align}
	\label{bump2}
	f_{b.o.t}(v)=\frac{n_p}{\sqrt{2\pi}}\exp(-\frac{v^2}{2})+\frac{n_b}{\sqrt{2\pi}}\exp(-\frac{(v-v_b)^2}{2v_t^2}),
	\end{align}
	and $n_p=0.9$, $n_b=0.2$, $v_b=4.5$, $v_t=0.5$, $\alpha=0.04$, $k=0.3$. The computational domain is $[0,\frac{2\pi}{k}]\times[-3\pi, 3\pi]$. We take
	the cut-off domain in $v$ a little larger for the consideration of the shifting $v_b$ in the velocity. For this problem, the minimum value of the initial data is taking $v=-3\pi$ and $\cos(kx)=-1$ in \eqref{bump} and \eqref{bump2}, while the maximum value is taking $v=0$ and $\cos(kx)=1$.
	We run a long time simulation up to $T=1000$, which is a good test to show the effectiveness of the linear scheme in preserving some important quantities and the saving of computational cost as compared to the WENO type scheme. We first show the surface of the distribution function at $T=500$ in Fig. \ref{fig26}, we can clearly observe without limiter, the solution has undershootings, while the scheme with limiter produces
	very good result. We show the time evolution of the electric field in $L^2$ norm and $L^\infty$ norm in Fig. \ref{fig27}, and compare the linear scheme with the WENO type scheme. The $L^\infty$ norm of the electric filed $E_{max}$ matches those in \cite{xiong2014high,arber2002critical}. However, we can see that the linear scheme preserves the electric field better than the WENO type scheme after some large time, e.g., $T=400$,  while the results are almost the same between with and without limiter.  We then show the relative derivation of the discrete $L^1$ norm, $L^2$ norm, kinetic energy and entropy in Fig. \ref{fig28}. The MPP limiter almost does not affect the $L^2$ norm and entropy, however, again we can see it improves the $L^1$ norm and the energy a lot from  eliminating negative numerical values, which indicates the great performance of the MPP limiter on the linear scheme. Besides, the linear scheme preserves the $L^2$ norm, entropy and energy much better than the WENO type scheme, especially the energy. Moreover, for the computational cost, the linear scheme is at least about 3-4 times faster than the WENO type scheme. This example with long time simulation clearly demonstrates that the linear scheme is better than the WENO type scheme, which is much less dissipative and computationally much more efficient.
		
	\begin{figure}
		\centering
		\includegraphics[width=3.2in]{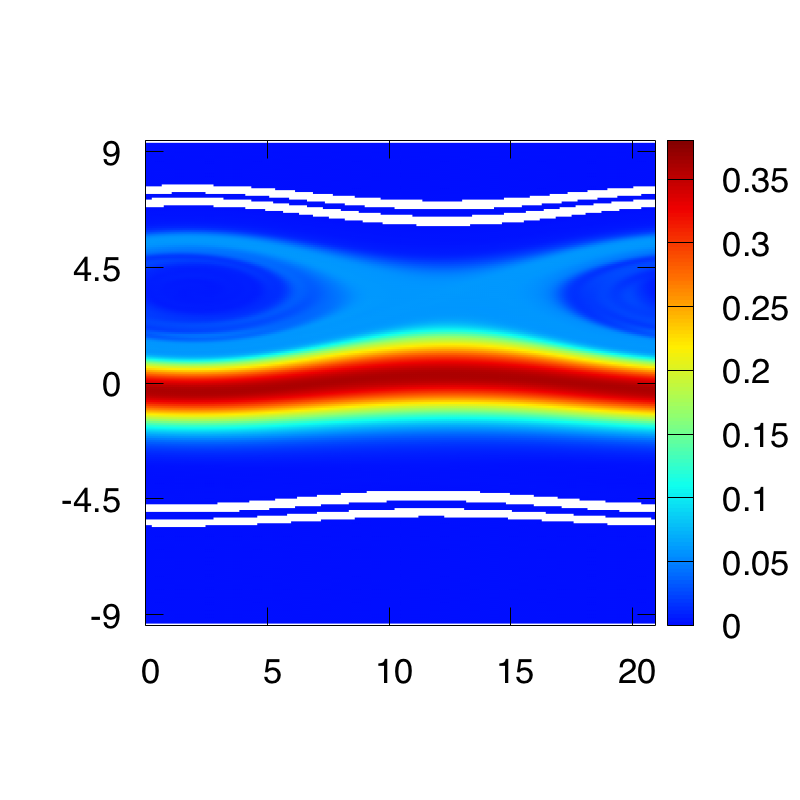}
		\includegraphics[width=3.2in]{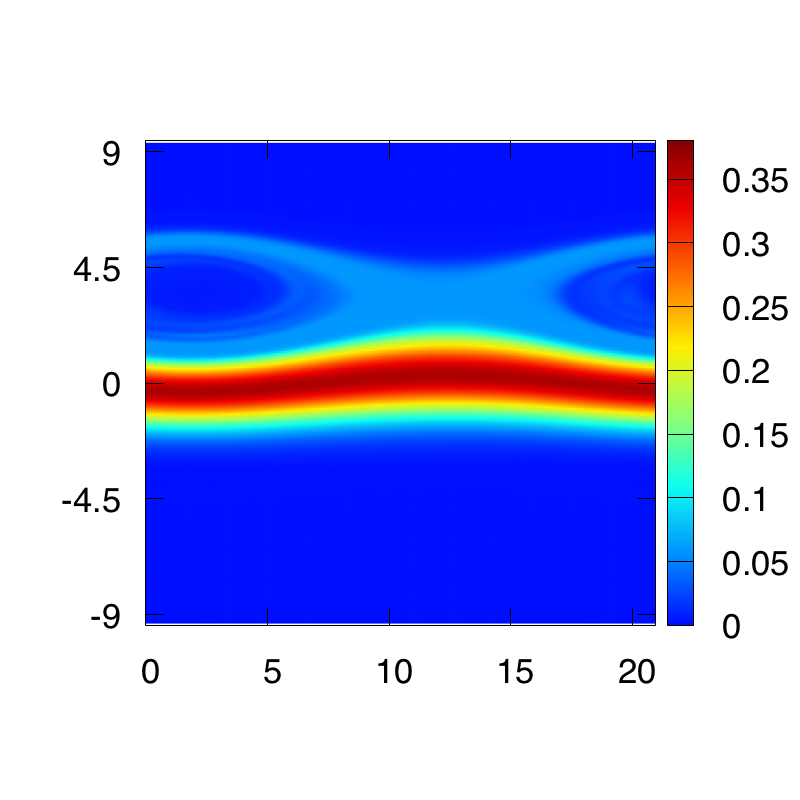}
		\caption{Bump-on-tail instability at $T=500$.
			Left: without limiter; Right: with limiter. Mesh grid: $256 \times 256$. }
		\label{fig26}
	\end{figure}

\begin{figure}
	\centering
	\includegraphics[width=5in,clip]{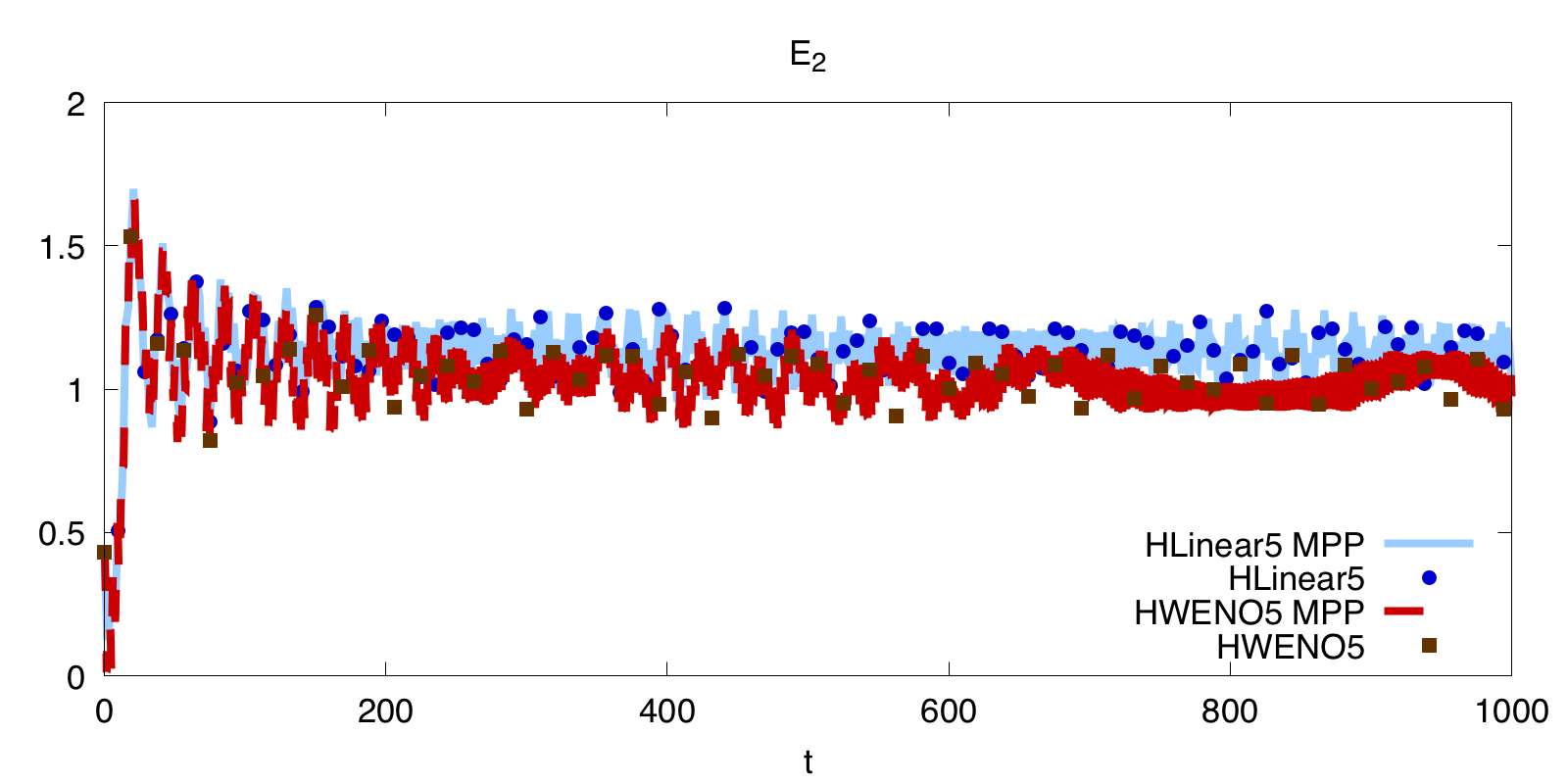}\\
	\includegraphics[width=5in,clip]{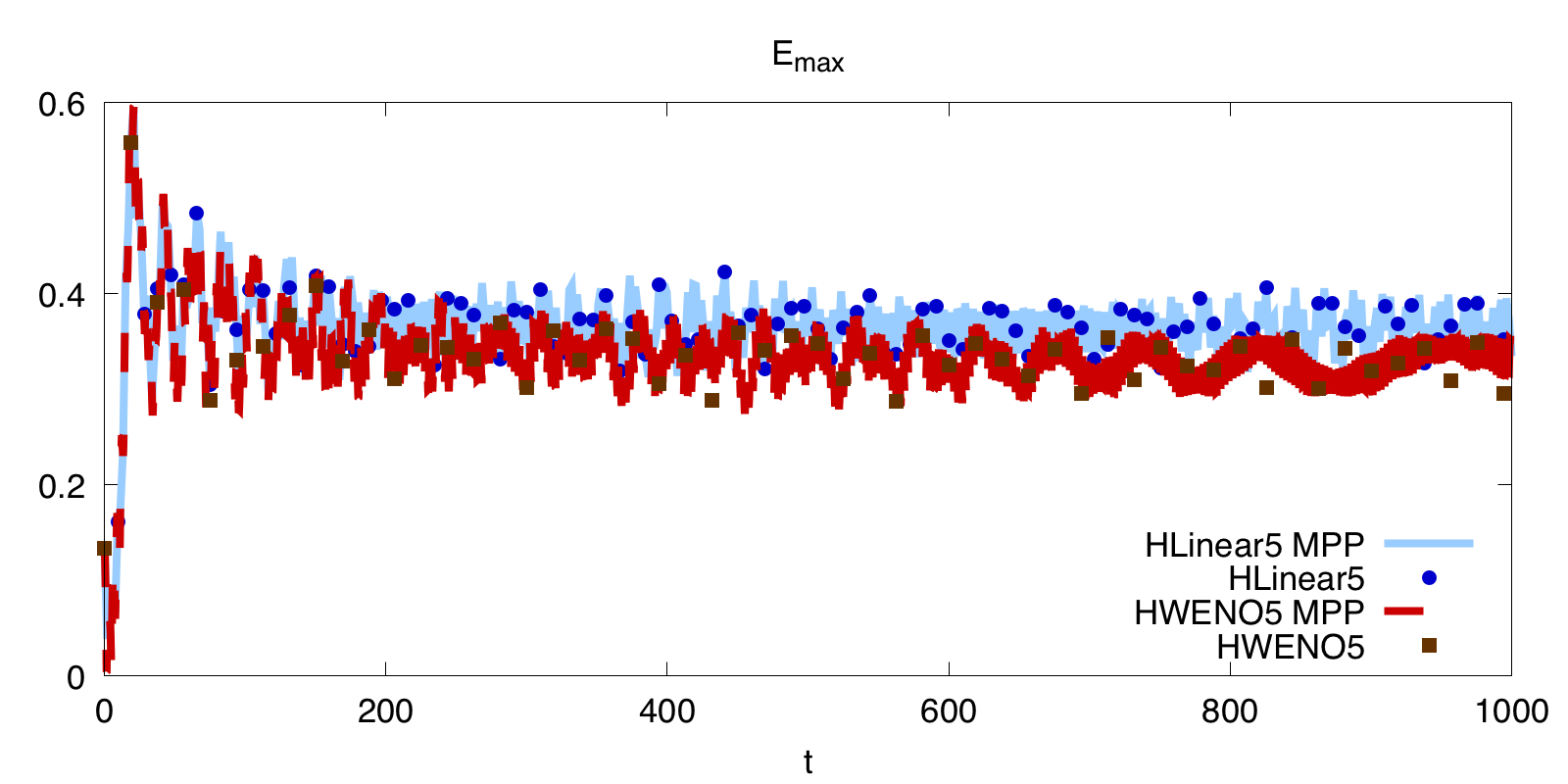}
	\caption{Bump-on-tail instability. Time evolution of the electric field in $L^2$ norm (top) and $L^\infty$ norm (bottom). Mesh grid: $256 \times 256$.}
	\label{fig27}
\end{figure}
	
\begin{figure}
	\centering
	%\includegraphics[width=3in,clip]{./pic2/VP/cut/Bot/_e2.png},\\
	%includegraphics[width=3in,clip]{./pic2/VP/cut/Bot/_emax.png}, \\
	\includegraphics[width=3in,clip]{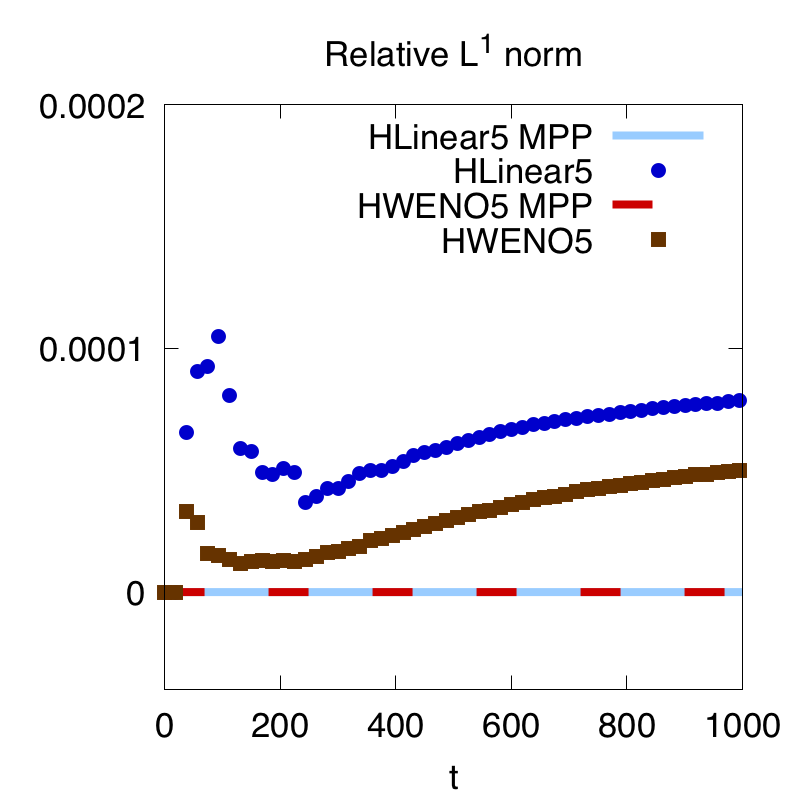}
	\includegraphics[width=3in,clip]{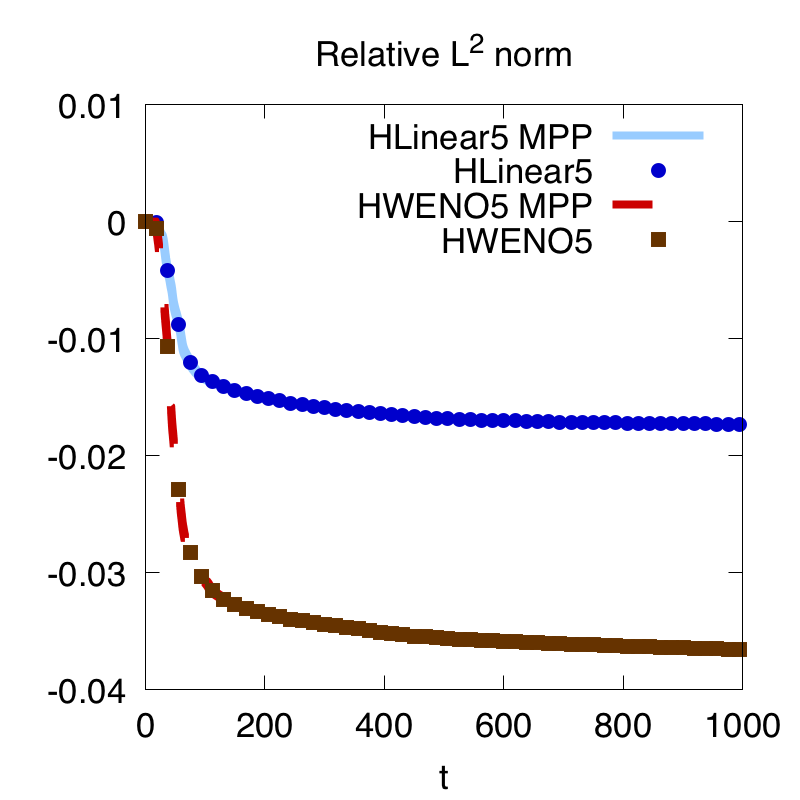} \\
	\includegraphics[width=3in,clip]{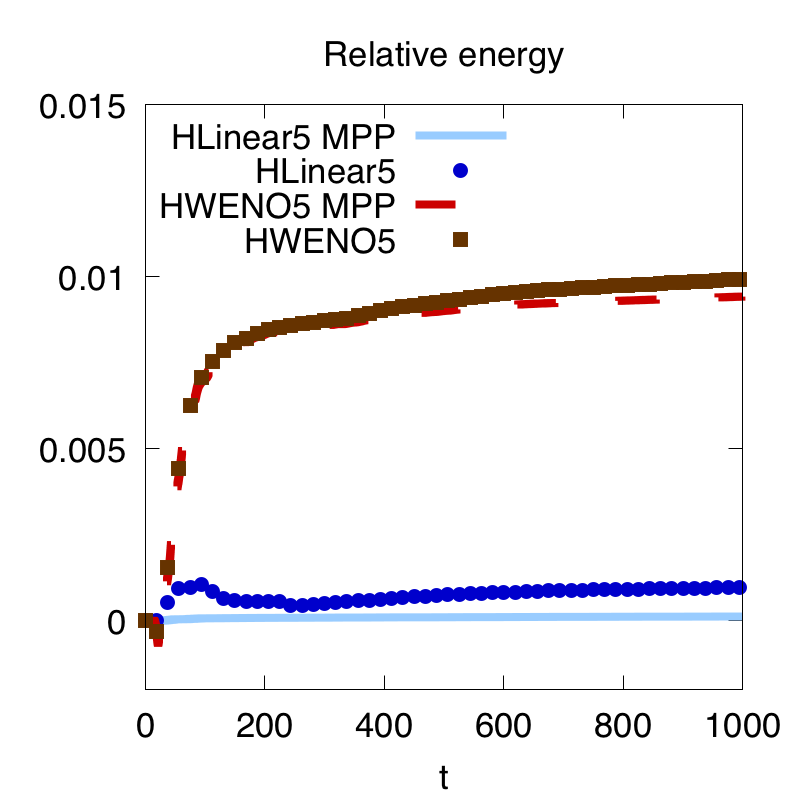}
	\includegraphics[width=3in,clip]{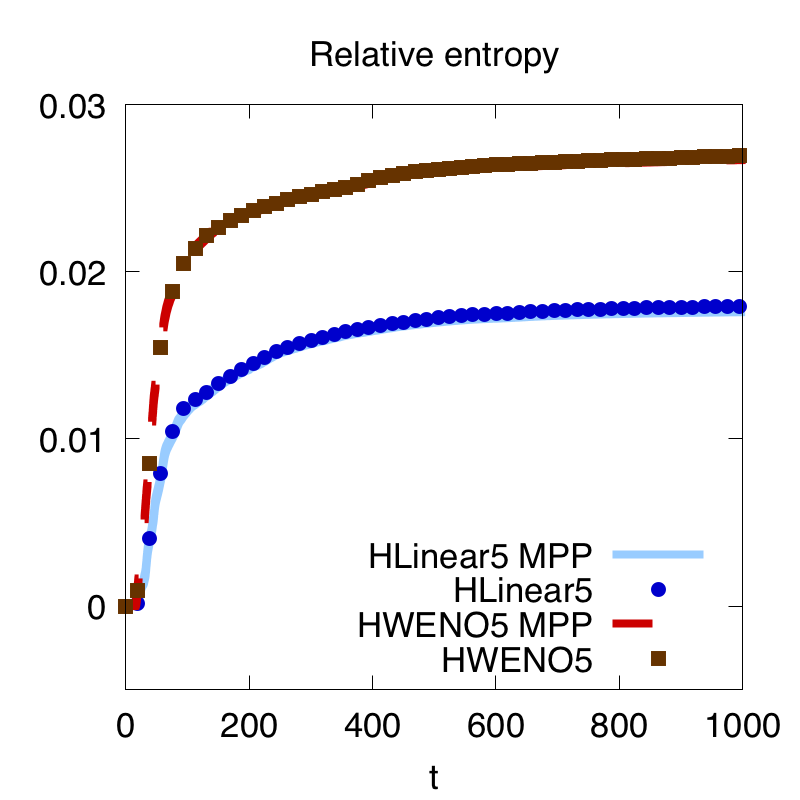}
	\caption{Bump-on-tail instability. Time evolution of the relative differences of discrete $L^1$ norm and $L^2$ norm (top), relative differences of kinetic energy and entropy (bottom). Mesh grid: $256 \times 256$.}
	\label{fig28}
\end{figure}
	
	\end{exa}

\subsection{Kelvin-Helmholtz instability}
%The 2D guiding-center model describes the time evolution of the charge density in a highly magnetized plasma in the transverse plane of a tokamak, which is given by
%\begin{equation}
%\frac{\partial \rho}{\partial t} + {\bf U} \cdot \nabla \rho = 0, \label{eq: guidingcenter}
%\end{equation}
%and
%\begin{equation}
%-\Delta \Phi =\rho,  \label{eq: poisson}
%\end{equation}
%where ${\bf U} = (-\partial_y \Phi, \partial_x \Phi)$ is the divergence free velocity.

\begin{exa} The Kelvin-Helmholtz instability comes from the 2D guiding-center model \eqref{eq:2dgc}  \cite{frenod2015long} with the initial data
	\begin{equation}
	\rho_0(x,y)=\sin(y)+0.015\cos(kx)
	\end{equation}
	and periodic boundary conditions on the domain $[0,4\pi]\times[0,2\pi]$. We let $k=0.5$, which
	will create a Kelvin-Helmholtz instability. The exact solution should be within the range $[-1.015, 1.015]$.
	
	For this example, we show the solution at $T=40$ with mesh grid $256\times 256$ for ``HLinear5" and ``HLinear5 MPP" in Fig. \ref{fig31} respectively. For the figures drawing in the physical range $[-1.015, 1.015]$, we can observe white spots for ``HLinear5" which is without MPP limiter, while ``HLinear5 MPP" with MPP limiter preserves the bounds very well.  In Fig. \ref{fig32}, we compare the time evolution of the relative $L^2$ norm and the numerical minimum values for both linear and WENO type schemes without and with limiters respectively. First, we can find that the linear scheme is less dissipative than the WENO one, as the $L^2$ norm preserves better for the linear scheme.  The schemes with and without MPP limiter preserves the $L^2$ norm almost the same, so that there is no significant affection from the MPP limiter. For the minimum numerical values, the linear scheme without MPP limiter has shown large undershootings, the WENO type scheme without MPP limiter performs even worse for this example (which is also similar for the classic WENO scheme ``WENO5", we omit the figure here), but both schemes with MPP limiter preserve the lower bound very well.  Similarly for the upper bound.  
	
	From this example, we can see that for problems with highly oscillatory but
	without discontinuous solutions, a high order linear scheme with the MPP limiter can result in a very good scheme, which is less dissipative and without significant spurious oscillations. The original idea to apply WENO reconstruction to suppress numerical oscillations seems to be even fail for this example. We can take this example as a benchmark test to support our main idea in this paper.
	
	\begin{figure}
		\begin{center}
			\includegraphics[width=3.2in]{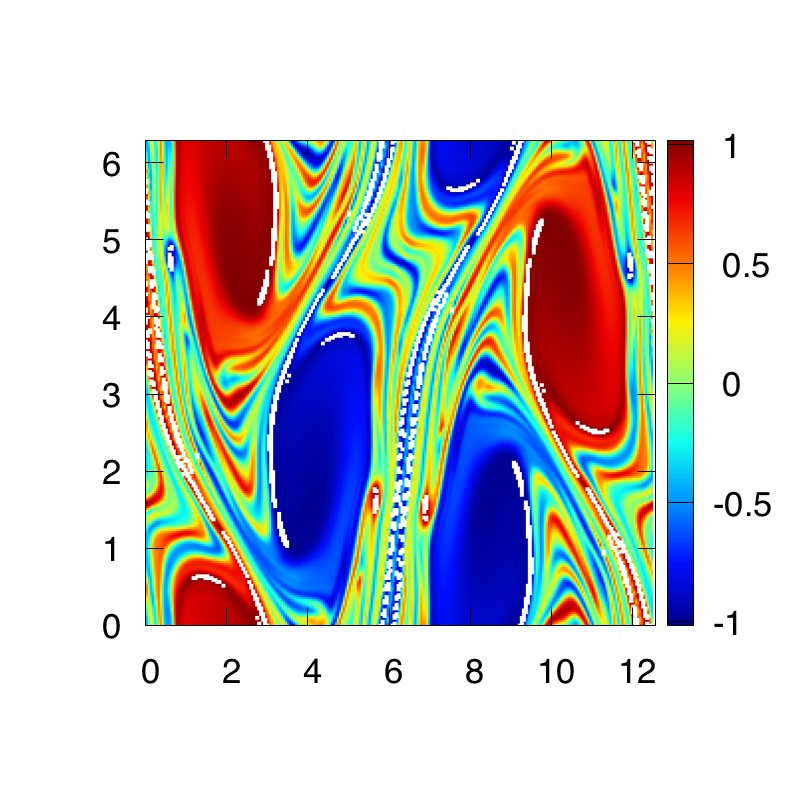}
			\includegraphics[width=3.2in]{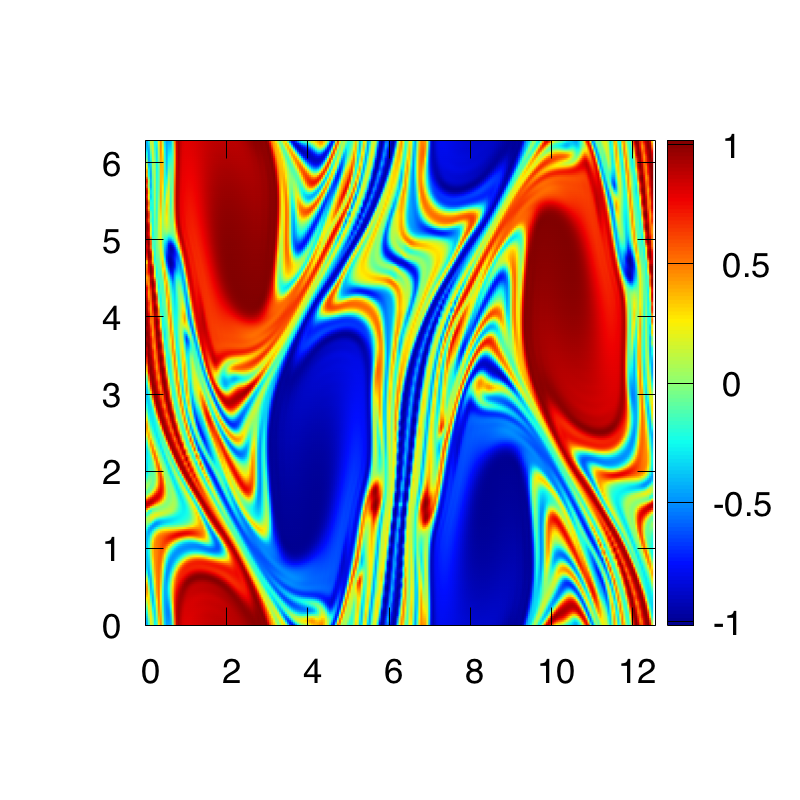}
		\end{center}
		\caption{Kelvin-Helmholtz instability problem at $T=40$. Mesh grid: $256\times256$.  Left: ``HLinear5"; Right: ``HLinear5 MPP".}
		\label{fig31}
	\end{figure}
	
	\begin{figure}
		\begin{center}
			\includegraphics[width=3.2in]{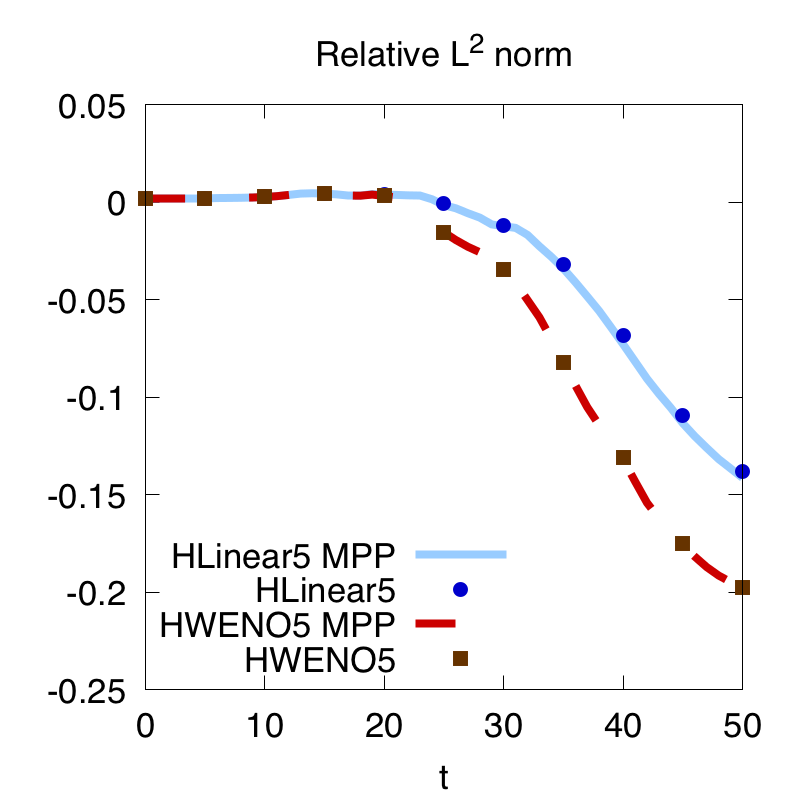}
			\includegraphics[width=3.2in]{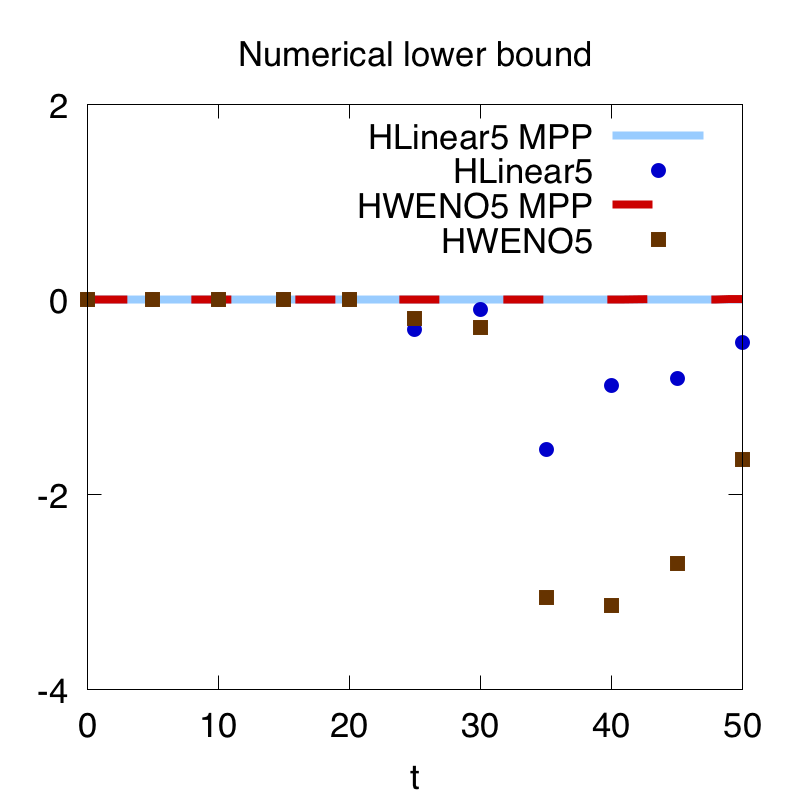}
		\end{center}
		\caption{Kelvin-Helmholtz instability problem. Mesh grid: $256\times256$. Left: relative $L^2$ norm; right: numerical minimum bound ($\rho_m+1.015$).}
		\label{fig32}
	\end{figure}
	
\end{exa}

\subsection{Incompressible Euler equations}
%The incompressible Euler equation is almost in the same form as the 2D guiding-center model problem,
%but with a different direction in the Poisson's equation and a different background, which reads
%\begin{equation}
%\frac{\partial \omega}{\partial t} + {\bf U} \cdot \nabla \omega = 0, \label{eq: inc}
%\end{equation}
%and
%\begin{equation}
%\Delta \Phi =\omega,  \label{eq: poisson2}
%\end{equation}
%where ${\bf U} = (-\partial_y \Phi, \partial_x \Phi)$ is the divergence free velocity for incompressibility.
\begin{exa}
	\label{ex67}(Accuracy test)
	We first consider the incompressible Euler system on the domain $[0, 2\pi]\times[0,2\pi]$ with an initial condition $\omega_0(x,y)=-2\sin(x)\sin(y)$. The exact solution will stay stationary with $\omega(x,y,t)=-2\sin(x)\sin(y)$,
	which is in the range of $[-2, 2]$. For this example, from Table \ref{tab2}, we can see that the numerical solution with and without the limiter all stay in the right range, which shows that the limiter does not destroy the high order accuracy.
	
	\begin{table}[h]
		\centering
		\caption{Example \ref{ex67}. $\omega_{min}$ and $\omega_{max}$ are the minimum and maximum numerical solutions respectively. ``WO'' stands for ``without limiter'', ``WL'' stands
			for ``with limiter''. $T=1$.}
		\vspace{0.2cm}
		\begin{tabular}{|c|c|c|c|c|c|c|c|}
			\hline
			& N  & $L^1$ error &    order   & $L^\infty$ error & order & $\omega_{max}$ & $\omega_{min}$ \\\hline
			\multirow{5}{*}{WO}
			&32 & 2.34e-05 &  --  & 4.89e-05 &  --  & 2 & -2   \\ \cline{2-8}
			&64 & 8.86e-07 & 4.72 & 1.63e-06 & 4.91 & 2 & -2 \\ \cline{2-8}
			&128 & 2.93e-08 & 4.92 & 5.08e-08 & 5.00 & 2 & -2 \\ \cline{2-8}
			&256 & 9.43e-10 & 4.96 & 1.60e-09 & 4.99 & 2 & -2  \\ \hline
			\multirow{5}{*}{WL}
			&32 & 2.34e-05 &  --  & 4.89e-05 &  --  & 2 & -2  \\ \cline{2-8}
			&64 & 8.86e-07 & 4.72 & 1.63e-06 & 4.91 & 2 & -2  \\ \cline{2-8}
			&128 & 2.93e-08 & 4.92 & 5.08e-08 & 5.00 & 2 & -2 \\ \cline{2-8}
			&256 & 9.43e-10 & 4.96 & 1.60e-09 & 4.99 & 2 & -2\\ \hline
		\end{tabular}
		\label{tab2}
	\end{table}
	
\end{exa}

\begin{exa} (Vortex patch).
	In this example, we consider the incompressible Euler equations for the vortex
	patch problem with the initial condition given by
	\begin{equation}
	\omega_0(x,y)=
	\begin{cases}
	-1, \qquad & \frac{\pi}{2} \le x \le \frac{3\pi}{2},  \frac{\pi}{4}\le y \le \frac{3\pi}{4}; \\
	1, \qquad & \frac{\pi}{2} \le x \le \frac{3\pi}{2}, \frac{5\pi}{4} \le y \le \frac{7\pi}{4}; \\
	0, \qquad & \text{otherwise}.
	\end{cases}
	\end{equation}
	We show the surface of $\omega$ at $T=10$ in Fig. \ref{fig41}. The mesh grid is $256\times256$. We can observe the good performance of the MPP flux limiter
	on this problem. We also show the time evolution of the relative difference of $L^2$ norm
	as compared to the initial data, and the minimum numerical solution on the bottom of Fig. \ref{fig41}.  Here we can still see that the linear scheme preserves the $L^2$ norm better than
	the WENO type scheme and the MPP limiter can eliminate the oscillations around two extreme values $1$ and $-1$ from the schemes without limiter. However, we would mention that for this example
	with discontinuous initial data, the MPP limiter cannot remove the small numerical oscillations around $0$ (figures are omitted here). In this case, the WENO will be needed to completely remove oscillations for middle discontinuities. 
	
	\begin{figure}
		\begin{center}
			\includegraphics[width=3.2in]{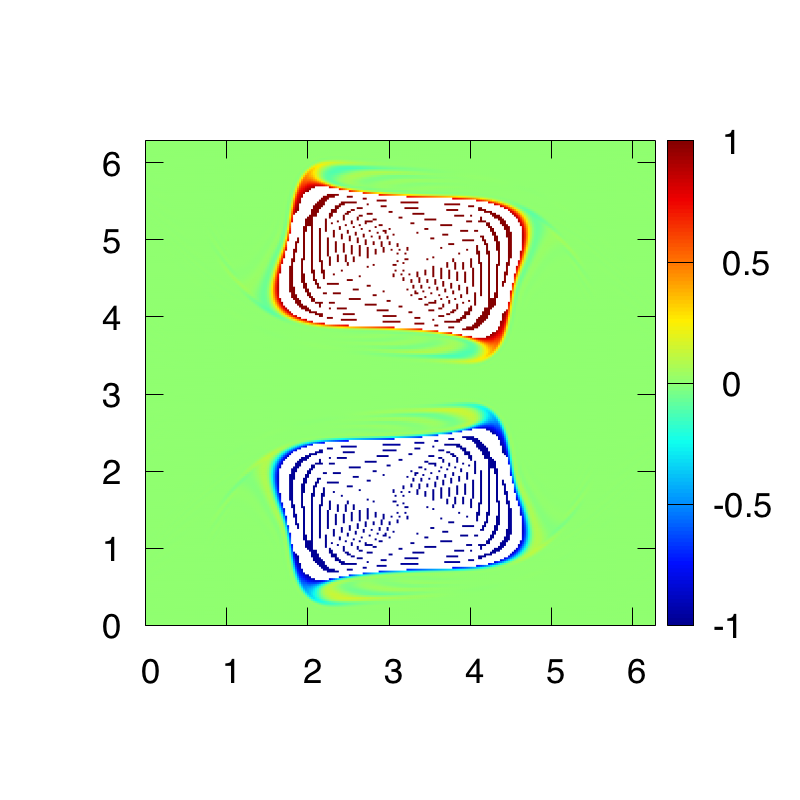}
			\includegraphics[width=3.2in]{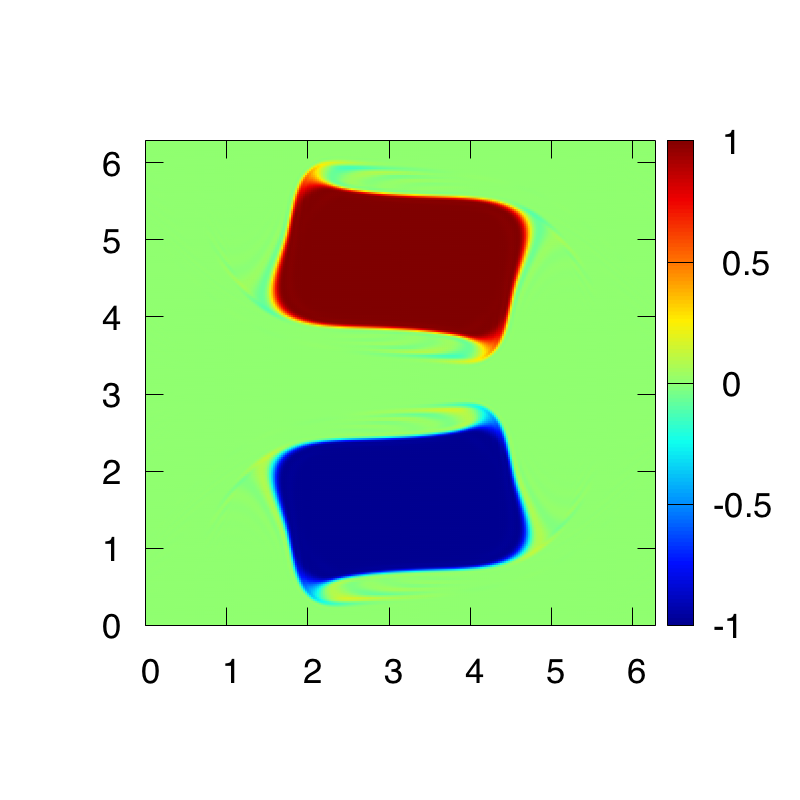}\\
			\includegraphics[width=3.2in]{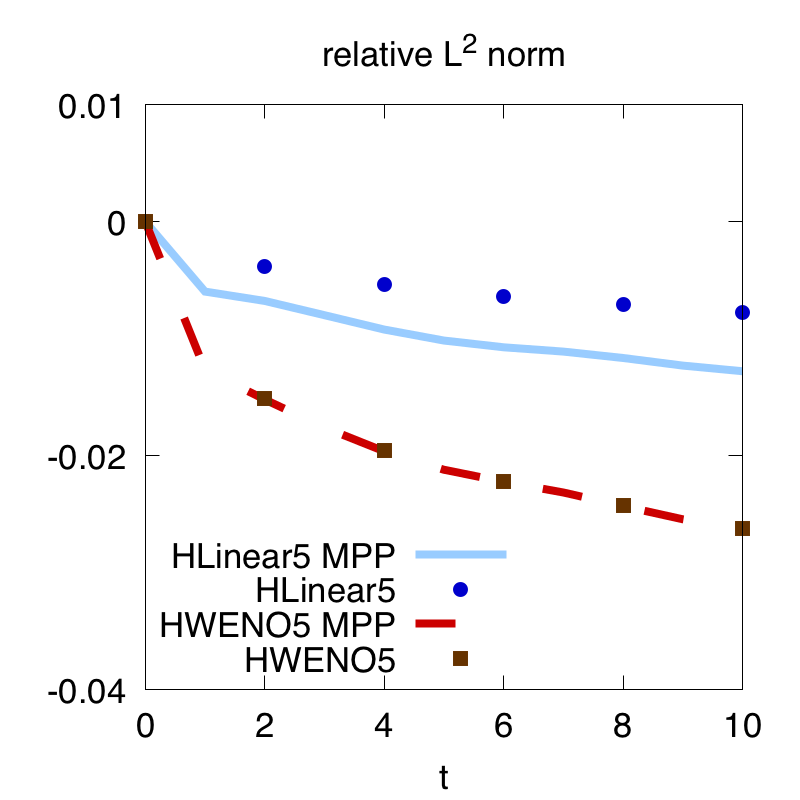}
			\includegraphics[width=3.2in]{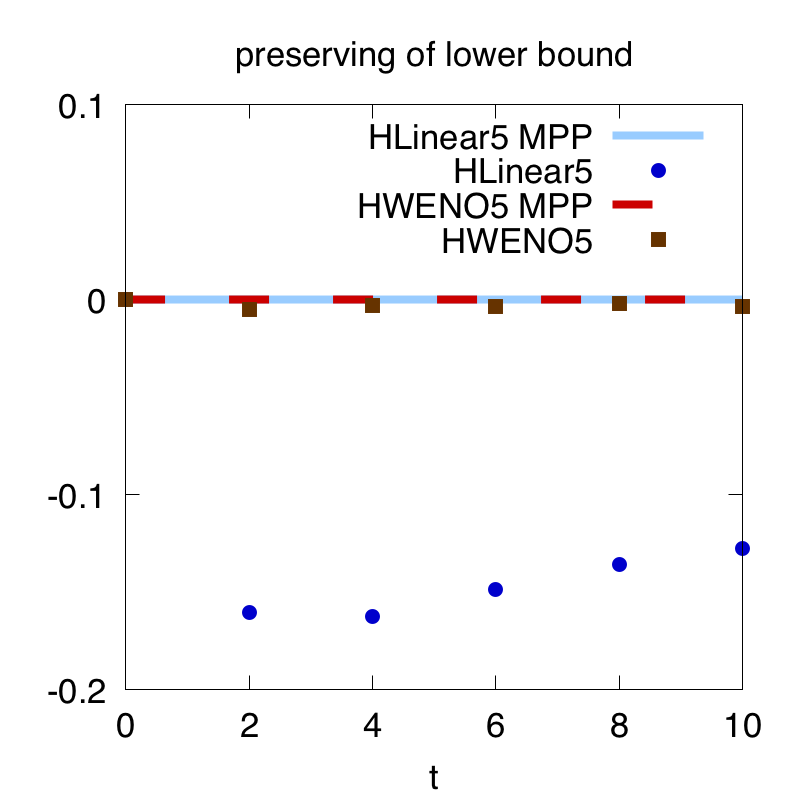}
		\end{center}
		\caption{The vortex patch problem. Mesh grid: $256\times256$. Top: surface of the vorticity $\omega$ at $T=10$ without and with MPP limiter on the left and right respectively. Bottom: the relative $L^2$ norm and the numerical minimum bound ( $\omega_m+1$). }
		\label{fig41}
	\end{figure}
	
\end{exa}
%
%\begin{exa} (Shear flow problem).
%This example is the same as above but with following
%initial conditions
%\begin{equation}
%\omega_0(x,y)=
%\begin{cases}
%\delta \cos(x)-\frac{1}{\rho}sech^2((y-\pi/2)/\rho)^2, \qquad &y\le \pi; \\
%\delta \cos(x)+\frac{1}{\rho}sech^2((3\pi/2-y)/\rho)^2,\qquad & y>\pi.
%\end{cases}
%\end{equation}
%where $\delta=0.05$ and $\rho=\frac{\pi}{15}$.
%We show the surface of $\omega$ at $T=6$ and $8$ in Fig. \ref{fig32}. The mesh size is $128\times128$.
%
%\begin{figure}
%\begin{center}
%\includegraphics[height=2.5in,width=3.0in]{./pic/incomp/shear6.png},
%\includegraphics[height=2.5in,width=3.0in]{./pic/incomp/shear8.png}
%\end{center}
%\caption{Shear flow problem. Mesh size $128\times128$. $T=6$ (left) and $T=8$ (right). }
%\label{fig32}
%\end{figure}
%
%\end{exa}

%%%%%%%%%%%%%%%%%%%
%
%%%%%%%%%%%%%%%%%%%
\section{Conclusion}
\label{sec6}
\setcounter{equation}{0}
\setcounter{figure}{0}
\setcounter{table}{0}

In this paper, we have proposed a linear scheme combined with a high order bound-preserving MPP flux limiter,
for solving incompressible flow problems. As compared to WENO type schemes, our approach
is less dissipative and much less costly, so that is much more efficient for high dimensional problems 
with long time simulations.  Applications to the Vlasov-Poisson system, 2D guiding-center model in
plasma physics, as well as the incompressible Euler equations in fluid hydrodynamics have demonstrated
the good performance of our proposed approach.

%%%%%%%%%%%%%%%%%%%
%
%%%%%%%%%%%%%%%%%%%

\section*{Acknowledgement}

%Francis Filbet and Eric Sonnendr\"ucker  were supported by the EUROfusion Consortium and has received funding
%from the Euratom research and training programme 2014-2018 under grant
%agreement No 633053. The views and opinions expressed herein do not
%necessarily reflect those of the European Commission.

T. Xiong would like to thank Dr. F. Filbet for many helpful discussions. T. Xiong acknowledges support from the Marie Sk{\l}odowska-Curie Individual Fellowships H2020-MSCA-IF-2014 of the European Commission, under the project HNSKMAP 654175. This work has also been supported by the Fundamental Research Funds for the Central Universities No. 20720160009, NSFC grant 11601455, NSAF grant U1630247 and NSF grant of Fujian Province 2016J05022. 

\bibliographystyle{siam}
\bibliography{refer}

%\begin{flushleft} \signFF \end{flushleft}
%%\vspace*{-44mm}
\begin{flushright} \signTX \end{flushright}
%%\vspace*{44mm}
%\begin{flushleft} \signES \end{flushleft}

\end{document}